\documentclass[12pt]{amsart}
\usepackage{amssymb}
\usepackage{amsmath}
\usepackage{amsthm}
\usepackage{mathrsfs}
\usepackage{MnSymbol}
\usepackage{stmaryrd}
\usepackage{enumitem}
\usepackage{graphicx}
\usepackage[all,cmtip]{xy}
\usepackage{tikz-cd}
\usepackage{multirow}
\usepackage{scalerel}
\usepackage{arydshln}
\usepackage{caption}
\usepackage{subcaption}
\usepackage{float}  
\usepackage{hyperref}
\usepackage[msc-links]{amsrefs}
\hypersetup{
  colorlinks   = true,	
  urlcolor     = blue,
  linkcolor    = blue,
  citecolor   = red
}
\newtheorem{theorem}{Theorem}
\newtheorem{lemma}[theorem]{Lemma}
\newtheorem{remark}[theorem]{Remark}

\newtheorem{proposition}[theorem]{Proposition}
\numberwithin{theorem}{section} \numberwithin{equation}{section}
\newtheorem*{condition}{Conditions for $L$}
\newcommand{\beq}{\begin{small} \begin{equation}}
\newcommand{\eeq}{\end{equation} \end{small}}
\newcommand{\beqn}{\begin{small} \begin{equation*}}
\newcommand{\eeqn}{\end{equation*} \end{small}}
\DeclareMathAlphabet{\mathpzc}{OT1}{pzc}{m}{it}
\newcommand\omicron{o}
\newcommand\scalemath[2]{\scalebox{#1}{\mbox{\ensuremath{\displaystyle #2}}}}
\newcommand*{\MyScalePicBig}{0.3}

\newcommand*{\MyScalePicTiny}{0.11}
\newcommand{\upper}[1]{\raisebox{2pt}{\scriptsize$#1$}}

\newcommand{\ti}{\widetilde}

\newcommand{\FF}{\mathbb{F}}

\newcommand{\RR}{\mathbb{R}}

\newcommand{\ZZ}{\mathbb{Z}}

\newcommand{\PP}{\mathbb{P}}

\newcommand{\calC}{\mathcal{C}}

\newcommand{\calQ}{\mathcal{Q}}

\newcommand{\calK}{\mathcal{K}}

\newcommand{\calW}{\mathcal{W}}
\newcommand{\calN}{\mathcal{N}}
\newcommand{\calS}{\mathcal{S}}

\newcommand{\calY}{\mathcal{Y}}
\newcommand{\calX}{\mathcal{X}}
\newcommand{\calZ}{\mathcal{Z}}

\newcommand{\Aut}{\mathop{\mathrm{Aut}}}

\calclayout
\allowdisplaybreaks
\begin{document}
\title[On two-elementary K3 surfaces with finite automorphism group]{On two-elementary K3 surfaces with finite automorphism group}
\author{Adrian Clingher}
\address{Department of Mathematics and Statistics, University of Missouri - St. Louis, St. Louis, MO 63121}
\email{clinghera@umsl.edu}
\author{Andreas Malmendier}
\address{Department of Mathematics \& Statistics, Utah State University, Logan, UT 84322}
\email{andreas.malmendier@usu.edu}
\author{Flora Poon}
\address{National Center for Theoretical Sciences Mathematics Division, National Taiwan University, Taipei 106319, Taiwan (R.O.C.)}
\email{wkpoon@ncts.ntu.edu.tw}
\begin{abstract}
We study complex algebraic K3 surfaces of Picard ranks $11, 12$, and $13$ of finite automorphism group that admit a Jacobian elliptic fibration with a section of order two. We prove that the K3 surfaces admit a birational model isomorphic to a projective quartic hypersurface and construct geometrically the frames of all supported Jacobian elliptic fibrations.  We determine the dual graphs of all smooth rational curves for these K3 surfaces, the polarizing divisors, and the embedding of the reducible fibers in each frame into the corresponding dual graph.
 \end{abstract}
\keywords{K3 surfaces, elliptic fibrations, Nikulin involutions}
\subjclass[2020]{14J27, 14J28}
\maketitle
%
%
 \section{Introduction}\label{sec_intro}
Let $\mathcal{X}$ be a smooth algebraic K3 surface defined over the field of complex numbers. Denote by $\operatorname{NS}(\mathcal{X})$ the N\'eron-Severi lattice of $\mathcal{X}$. 
Let $L$ be a choice of even indefinite lattice of rank $\rho_L$ and signature $(1,\rho_L-1)$ where $ 1 \leq \rho_L \leq 19$. 
As in \cite{MR4635303}*{Section 2B}, fix $h$ to be a very irrational vector in $L\otimes \RR$ of positive norm, where very irrational means that $h\notin L' \otimes \RR$ for any primitive proper sublattice $L' \subsetneq L$.
Then, following \cite{MR4635303}*{Definition 2.6}, a {\it lattice polarization} on $\mathcal{X}$ is, by definition, a choice of a primitive lattice embedding $i \colon L \hookrightarrow \operatorname{NS}(\mathcal{X})$, such that $i(h)$ is big and nef. Two $L$-polarized K3 surfaces $(\mathcal{X},i)$ and $(\mathcal{X}',i')$ are said to be isomorphic  if there exists an isomorphism $\alpha \colon \mathcal{X} \rightarrow \mathcal{X}'$ and a lattice isometry $\beta \in O(L)$ such that $ \alpha^* \circ i' = i \circ \beta$, where $\alpha^*$ is the induced morphism at cohomology level.
\par It is known \cite{MR1420220} that $L$-polarized K3 surfaces are classified, up to isomorphism, by a coarse moduli space $\mathscr{M}_L$, which is a quasi-projective variety of dimension $20-\rho_L$.  A \emph{very general} $L$-polarized K3 surface $(\mathcal{X},i)$ satisfies $i(L)=\operatorname{NS}(\mathcal{X})$. In particular, the Picard number of $\calX$ is $\rho_L$.
\par If the standard rank-two hyperbolic lattice $H$ primitively embeds into $L$, then a $L$-polarized K3 surface $\calX$ has a Jacobian elliptic fibration, which is an elliptic fibration $\pi_\mathcal{X} \colon \mathcal{X} \rightarrow \mathbb{P}^1 $ with a choice of a section \cite{MR2369941}*{Lemma~3.6}.
If furthermore, the negative definite rank-eight Nikulin lattice \cite{MR728142}*{Def.~5.3} that we denote by $N$, also primitively embeds into $L$, 
then the  Mordell-Weil group $\operatorname{MW}(\mathcal{X}, \pi_\mathcal{X})$ associated to the elliptic fibration $\pi_\mathcal{X}$
has an element of order-two  \cite{MR2274533}.
This element can be seen geometrically as a second section in the elliptic fibration $\pi_\mathcal{X}$, which determines an order-two point in each smooth fiber. Fiber-wise translation by these order-two points further determines a canonical symplectic involution known as a \emph{van Geemen-Sarti involution} (see Section~\ref{subsec_K3 surfaces with van Geemen-Sarti involution}). 

\par In this article, we study families of lattice polarized K3 surfaces whose polarization lattice $L$ satisfies the following conditions (see Section \ref{subsec_elliptic fibrations of L-polarised K3 surfaces} for details).
\begin{condition}
~\\
\vspace{-0.5cm}
\begin{enumerate}[label=(\alph*)]
\item $L \simeq H \oplus K$, with $K$ a negative-definite lattice of ADE type,\label{cond_a}
\item $L$ admits a unique primitive lattice embedding $H \oplus N \hookrightarrow L$, \label{cond_b}
\item $L$ is of finite automorphism group type, in the sense of Nikulin \cites{MR3165023,MR633160b}.  \label{cond_c}
\end{enumerate} 
\end{condition}

\par 
By standard classification results on lattices \cite{MR633160b} and elliptic fibrations \cites{MR1813537, MR4130832}, one can exhaust all the lattices $L$ satisfying conditions (a)-(c) up to isometry, which are displayed in column 2 in Table \ref{tab:extension}.
\begin{table}[!ht]
\scalemath{0.7}{
\begin{tabular}{c|c|lclll}
$\begin{array}{c} \rho_L \\ {\scriptstyle (\ell_L, \delta_L)} \end{array}$		&  \multicolumn{1}{c|}{$L$}  & \multicolumn{1}{c}{$K^{\text{root}}$} &  \multicolumn{1}{l}{$\qquad \mathpzc{W}$} 	& construction  	& dual graph& embedding\\
\hline
\hline
$10$				& $H(2)\oplus D_4(-1)^{\oplus 2}\cong H \oplus N$ 			& $8 A_1$ 		& $\mathbb{Z}/2\mathbb{Z}$ 		&  \; \cite{MR2274533}	& \cites{MR2274533, Roulleau22}& \\
\upper{(6, 0)} 		&														& $2 D_4$		& no section					&  \; \cites{MR4444083,MR4069236,MR1029967} & \\
\hline
$11$ 			& $H\oplus D_4(-1) \oplus A_1(-1)^{\oplus 5}$ 						& $D_4 + 5 A_1$ 	& $\{ \mathbb{I} \}$			& \multirow{3}{*}{\scalerel*[1ex]{\}}{\rule[-3ex]{1ex}{7ex}}\, Thm.~\ref{thm_11}}				& \multirow{3}{*}{Fig.~\ref{fig:pic11_2folds}} & \multirow{3}{*}{Thm.~\ref{thm:polarization11}(5)}\\
\upper{(7, 1)}		& 														& $9 A_1$ 		& $\mathbb{Z}/2\mathbb{Z}$	\\	
				&														& $2 D_4 + A_1$	& no section				\\
\hline
$12$ 			& $H\oplus D_6(-1) \oplus A_1(-1)^{\oplus 4}$ 						& $D_6 + 4 A_1$ 	& $\{ \mathbb{I} \}$			& \multirow{3}{*}{\scalerel*[1ex]{\}}{\rule[-3ex]{1ex}{7ex}}\, Thm.~\ref{thm_12}}				& \multirow{3}{*}{Fig.~\ref{fig:pic12_2folds}} & \multirow{3}{*}{Thm.~\ref{thm:polarization12}(5)}\\
\upper{(6,1)} 		& $\quad \cong H \oplus D_4(-1)^{\oplus 2}  \oplus A_1(-1)^{\oplus 2}$	& $2 D_4 + 2 A_1$ 	& $\{ \mathbb{I} \}$			\\
&																		& $D_4 + 6 A_1$ 	& $\mathbb{Z}/2\mathbb{Z}$	\\
\hline
$13$ 			& $H\oplus E_7(-1) \oplus A_1(-1)^{\oplus 4}$ 						& $E_7 + 4 A_1$ 	& $\{ \mathbb{I} \}$			& \multirow{4}{*}{\scalerel*[1ex]{\}}{\rule[-4ex]{1ex}{9ex}}\, Thm.~\ref{thm_13}}				& \multirow{4}{*}{Fig.~\ref{fig:pic13_2folds}} & \multirow{4}{*}{Thm.~\ref{thm:polarization13}(5)}\\
\upper{(5,1)} 		& $\quad \cong H \oplus D_8(-1)  \oplus A_1(-1)^{\oplus 3}$ 			& $D_8 + 3 A_1$ 	& $\{ \mathbb{I} \}$			\\
				& $\quad \cong H \oplus D_6(-1)  \oplus D_4(-1) \oplus A_1(-1)$ 		& $D_6 + D_4 + A_1$ & $\{ \mathbb{I} \}$			\\
&																		& $D_6 + 5 A_1$ 	& $\mathbb{Z}/2\mathbb{Z}$	\\
\hline
$14$ 			& $H\oplus D_8(-1) \oplus D_4(-1)$ 								& $D_8 + D_4$ 	& $\{ \mathbb{I} \}$				& \multirow{2}{*}{\scalerel*[1ex]{\}}{\rule[-2ex]{1ex}{5ex}}\, \cite{Clingher:2020baq}}			& \multirow{2}{*}{\cites{Clingher:2020baq, MR1029967}} \\
\upper{(4,0)} 		&														& $E_7 + 5 A_1$ 	& $\mathbb{Z}/2\mathbb{Z}$	\\
\hline
$14$ 			& $H\oplus E_8(-1) \oplus A_1(-1)^{\oplus 4}$ 						& $E_8 + 4 A_1$ 	& $\{ \mathbb{I} \}$					& \multirow{5}{*}{\scalerel*[1ex]{\}}{\rule[-5ex]{1ex}{11ex}}\, \cite{Clingher:2020baq}}			& \multirow{5}{*}{\cites{Clingher:2020baq, MR1029967}} \\
\upper{(4,1)} 		& $\quad \cong H \oplus D_{10}(-1) \oplus A_1(-1)^{\oplus 2}$ 			& $D_{10} + 2 A_1$ 	& $\{ \mathbb{I} \}$			\\
				& $\quad \cong H \oplus E_7(-1) \oplus D_4(-1) \oplus A_1(-1)$			& $E_7 + D_4 + A_1$& $\{ \mathbb{I} \}$			\\
				& $\quad \cong H \oplus D_{6}(-1)^{\oplus 2}$ 						& $2 D_6$ 		& $\{ \mathbb{I} \}$			\\
				&														& $D_8 + 4 A_1$ 	& $\mathbb{Z}/2\mathbb{Z}$	\\
\hline
$15$				& $H \oplus D_{12}(-1) \oplus A_1(-1)$ 							& $D_{12} + A_1$ 	& $\{ \mathbb{I} \}$			& \multirow{4}{*}{\scalerel*[1ex]{\}}{\rule[-4ex]{1ex}{9ex}}\, \cite{Clingher:2020baq}}			& \multirow{4}{*}{\cites{Clingher:2020baq, MR1029967, MR633160b}} & \\
\upper{(3,1)} 		& $\cong  H \oplus E_8(-1)  \oplus D_4(-1) \oplus A_1(-1)$ 			& $E_8 + D_4 + A_1$& $\{ \mathbb{I} \}$			\\			
				& $\cong H \oplus E_7(-1) \oplus D_6(-1)$ 						& $E_7 + D_6$ 	& $\{ \mathbb{I} \}$			\\
				&														& $D_{10} + 3 A_1$	&$\mathbb{Z}/2\mathbb{Z}$	\\
\hline
$16$				& $H \oplus D_{14}(-1)$ 										& $D_{14}$ 		& $\{ \mathbb{I} \}$			& \multirow{4}{*}{\scalerel*[1ex]{\}}{\rule[-4ex]{1ex}{9ex}}\, \cites{MR4544843, MR4160930}}							& \multirow{4}{*}{\cites{MR4544843, MR1029967, MR633160b}} \\
\upper{(2,1)}  		& $\cong H \oplus E_8(-1)  \oplus D_6(-1)$ 						& $E_8 + D_6$ 	& $\{ \mathbb{I} \}$			\\			
				& $\cong H \oplus E_7(-1) \oplus E_7(-1)$ 						& $2 E_7$ 		& $\{ \mathbb{I} \}$			\\	
				&														& $D_{12} + 2 A_1$ 	&$\mathbb{Z}/2\mathbb{Z}$	\\
\hline
$17$ 			& $H \oplus E_8(-1)  \oplus E_7(-1)$ 							& $E_8 + E_7$ 	& $\{ \mathbb{I} \}$				& \multirow{2}{*}{\scalerel*[1ex]{\}}{\rule[-2ex]{1ex}{5ex}}\, \cites{MR2427457, MR2824841, MR2935386}}				& \multirow{2}{*}{\cites{MR1029967, MR2824841, MR633160b}} &\\
\upper{(1,1)}		& 														& $D_{14} + A_1$ 	& $\mathbb{Z}/2\mathbb{Z}$	\\
\hline
$18$ 			& $H \oplus E_8(-1)  \oplus E_8(-1)$ 							& $2 E_8$ 		& $\{ \mathbb{I} \}$			& \multirow{2}{*}{\scalerel*[1ex]{\}}{\rule[-2ex]{1ex}{5ex}}\, \cites{MR2369941, MR2279280}}	& \multirow{2}{*}{\cites{MR2369941,MR1029967, MR633160b}} & \\
\upper{(0,0)} 		& 														& $D_{16}$ 		& $\mathbb{Z}/2\mathbb{Z}$	\\
\hline
\end{tabular}}
\captionsetup{justification=centering}
\caption{}
\label{tab:extension}
\end{table}

\par
As it turns out, there is one lattice $L$ satisfying (a)-(c) in each rank $11 \leq \rho_L \leq 18$, with the exception of rank $\rho_L = 14$, for which two inequivalent instances exist.  Note that the lattice $H\oplus N$ in the first row does not satisfy (a).  Nevertheless we have included the lattice in the table: each $L$-polarized K3 surface has a unique $H\oplus N$-polarization.  Moreover, the $H\oplus N$-polarized family is in fact a family of Jacobian elliptic surfaces (see Section {\color{blue} 3.1}), obtained from a different family of K3 surfaces of Picard rank $10$. From the latter family we will also derive birational models for the families of higher Picard rank in Table~\ref{tab:extension}.

Note that the rank $19$ lattice $L = H\oplus E_8(-1) \oplus E_8(-1)\oplus A_1(-1)$ is the only other lattice that is not shown in the table that satisfies conditions (a) and (c), and has a $H\oplus N$ embedding. However, the $H\oplus N$ embedding is not unique.

The topics of interest here are the geometric construction of the frames of the Jacobian elliptic fibrations for each $L$-polarized family and their dual graphs of rational curves: see columns $6$ and $7$ on Table \ref{tab:extension}. For those families of Picard rank at least $14$ and the $H\oplus N$-polarization family,  these questions were discussed in previous works by the authors and others.  The present paper fills in the gap for the remaining three families of Picard rank $11, 12$ and $13$.  Moreover, we present  an embedding of the irreducible fibers of each Jacobian elliptic fibration of these three families into the corresponding dual graphs by coloring their vertices: see column $8$ on Table \ref{tab:extension}.  
A summary of the results of this paper is displayed on rows $2$ to $4$ of Table \ref{tab:extension}.
\par The outline of the paper is as follows:  in Section~\ref{sec_background}, we give some background material important for this work. In particular, we explain the conditions satisfied by the lattice $L$. We also study the $H\oplus N$-polarized family in detail. In Section~\ref{sec_Z family}, we introduce a family of  K3 surfaces of Picard rank $10$ whose associated Jacobian surfaces gives the $H\oplus N$-polarized family. We also derive a projective quartic hypersurface model and an associated double sextic model. We show, in Section~\ref{sec_fams}, that these birational models can be specialized to models for the families of Picard rank $11, 12$ and $13$ in Table \ref{tab:extension}. There, we also introduce a normal form for the quartic hypersurface model as a multi-parameter generalization of the Inose quartic \cite{MR578868}, and construct the supported Jacobian elliptic fibrations from pencils of curves. In Section~\ref{sec_dual graph}, we determine the dual graphs of all smooth rational curves as well as the polarizing divisors for the constructed projective surfaces. We color the vertices of the dual graphs to denote embeddings of the irreducible fibers of the Jacobian elliptic fibrations. Symmetries of the colored dual graphs for each alternate fibration describe the underlying van Geeman-Sarti involution.

To our knowledge, for $L$-polarization with $\rho_L= 11, 12$, $13$, the quartic surfaces to be described in Section~\ref{sec_fams} have not appeared in the literature before. 
\section*{Acknowledgement}
The authors would like to thank Xavier Roulleau for pointing out an earlier inaccuracy in Section~\ref{sec_dual graph}. The authors also thank Alice Garbagnati, Cec\'ilia Salgado and Alan Thompson for helpful discussions. A.M. acknowledges support from the Simons Foundation through grant no.~202367. 
%
%
\section{Background}\label{sec_background}
\subsection{Elliptic fibrations of L-polarized K3 surfaces}\label{subsec_elliptic fibrations of L-polarised K3 surfaces}

Let us first recall the definitions of a few lattice invariants. 
Given a lattice $L$, we denote by $D(L) = L^\vee/L$ the associated \emph{discriminant group} with corresponding \emph{discriminant form} $q_L$.  A lattice $L$ is then called \emph{2-elementary} if $D(L)$ is a 2-elementary abelian group, i.e., $D(L) \cong (\mathbb{Z}/2\mathbb{Z})^\ell$. Here $\ell=\ell_L \in \ZZ_{\geq 0}$ is the {\it length} of $L$,  i.e., the minimal number of generators of the group $D(L)$.  As a third relevant invariant for $L$, one also has the {\it parity} $\delta_L \in \{ 0,1\}$. By definition, $\delta_{L} = 0$ if $q_L(x)$ takes values in $\mathbb{Z}/2\mathbb{Z} \subset \mathbb{Q}/2\mathbb{Z}$ for all $x \in D(L)$, and $ \delta_L=1$ otherwise.   
A result of  Nikulin \cite{MR633160b}*{Thm.~4.3.2} asserts that  hyperbolic, even, 2-elementary lattices embedding into the K3 lattice are uniquely determined by their rank $\rho_L$, length $\ell_L$, and parity $\delta_L$.

Let $\calX$ be a K3 surface with Neron-Severi lattice $L$ that satisfies Conditions \ref{cond_a}-\ref{cond_c} mentioned in Section \ref{sec_intro}. 
Condition \ref{cond_a} is equivalent to saying that $\calX$ carries a Jacobian elliptic fibration: 
the general fiber and section classes in $L$ generate a sublattice of hyperbolic type which gives a primitive lattice embedding $j\colon H\hookrightarrow L$.  
The \emph{frame} of the Jacobian elliptic fibration, which  is the isometry class of the lattice $K := j(H)^\perp$, is negative definite and has rank $\rho_L-2$. Moreover, all frames of $\mathcal{X}$ share the same discriminant group and discriminant form with $L$. Hence they belong to the same lattice genus, canonically associated with the surface $\mathcal{X}$. 
Let us also note that, given the frame $K$, the root sub-lattice $K^{{\rm root}}$ is always of ADE type and the factor group $\mathpzc{W} = K/K^{{\rm root}}$ is isomorphic to the Mordell-Weil group of the Jacobian elliptic fibration. The pair $(K^{{\rm root}}, \mathpzc{W})$ gives the \emph{type} of the frame $K$.
We shall refer to the elliptic fibration associated to the decomposition $L= 
 H \oplus K$ with trivial Mordell-Weil group $\calW$ as \emph{standard}.

The discussion of Condition \ref{cond_b}, which says that $\calX$ has a canonical Jacobian elliptic fibration with a two-torsion section, or $\calW = \ZZ/2\ZZ$, will be deferred to Section \ref{subsec_K3 surfaces with van Geemen-Sarti involution}.
The elliptic fibration associated with the $H$-embedding of this condition will be referred to as the \emph{alternate} fibration.

Lastly, we explain the meaning of Condition \ref{cond_c}: a lattice $L$ is said to be of \emph{finite automorphism type} if the factor group $ O(L)/W(L)$ is finite. 
Here $O(L)$ is the group of isometries of $L$ and $W(L)$ is the subgroup of isometries generated by the reflections with respect to roots of $L$. Following results by Pjatecki\u{\i}-\v{S}apiro and \v{S}afarevi\v{c} \cite{MR0284440} and Nikulin~\cite{MR633160b}, Condition \ref{cond_c} is in fact the same as the condition for the group of automorphism $\Aut(\calX)$ to be finite.


We note that for general $L$, a lattice-theoretic classification of Jacobian elliptic fibrations on $L$-polarized K3 surfaces with finite automorphism group was given by the authors in \cite{Clingher:2021}. 
This work was largely based on Nikulin's classical lattice classification theory \cite{MR633160b}, as well as on Shimada's work \cite{MR1813537} on Jacobian elliptic fibrations. 
The present paper focuses on K3 surfaces $\mathcal{X}$ with finite automorphism group that also carry a Jacobian elliptic fibration endowed with a 2-torsion section. By the classification in \cite{Clingher:2021}*{Tables~2 and 3}, this list consists of K3 surfaces with $2$-elementary N\'eron-Severi lattices only.  Moreover, one has:
\begin{theorem}
\label{thm:families}
Let $\mathcal{X}$ be a K3 surface satisfying Conditions \ref{cond_a}-\ref{cond_c} in Section \ref{sec_intro}. 
Then, up to isometry, the lattice $L$ belongs to the list given in Table~\ref{tab:extension}. Table~\ref{tab:extension} also includes the frames of all possible Jacobian elliptic fibrations supported on $\mathcal{X}$. 
\end{theorem}

\begin{remark}
Since each entry $\calX$ in Table~\ref{tab:extension} has automorphism group $\ZZ/2\ZZ \times \ZZ/2\ZZ$, the alternate fibration which comes from condition \ref{cond_b} is always unique up to automorphism of $\calX$. In contrast, there could be multiple inequivalent decompositions $L=H \oplus K$ with trivial Mordell-Weil groups.
\end{remark}

\begin{remark}\label{rmk_2.3}
    In the context of Condition \ref{cond_a}, an important question that was also studied in \cite{Clingher:2021} is - how many non-isomorphic Jacobian elliptic fibrations on $\mathcal{X}$ can be associated to a given frame. This number, called the \emph{multiplicity} of the frame, can be described via a lattice-theoretic procedure due to Festi and Veniani \cite{FestiVeniani20}. 
    In Table~\ref{tab:extension}, the multiplicity of the unique frame $(K^{\text{root}}, \mathpzc{W})$ with $\mathpzc{W}=\mathbb{Z}/2\mathbb{Z}$ for each lattice $L$ with $\rho_L >10$ is in fact $1$. 
    We direct interested readers to \cite{Clingher:2021}*{Section~5} for more details.
\end{remark}

Before we move on to studying explicit families of lattice polarized K3 surfaces and their elliptic fibrations, we would like to recall some facts about {\em Jacobian surfaces} of elliptic K3 surfaces \cite{Jac(S)}. Given an elliptic K3 surface $f\colon  \calX \rightarrow \PP^1$, the Jacobian surface of $\calX$ is the elliptic surface equipped with the {\em relative Jacobian fibration} $J(\calX) \rightarrow \PP^1$ whose fiber at a point $z \in \PP^1$ is $J(f_z)$, the Jacobian of the fiber of $f$ at $z$.  In the absence of any multiple fibers, $J(\calX)$ has the same Picard number and singular fibers of the same types as $\calX$. In this paper, very often we would consider both the K3 surface and its Jacobian surface for the analysis of singular fibers of an elliptic fibration.
%
%
\subsection{K3 surfaces with van Geemen-Sarti involution}\label{subsec_K3 surfaces with van Geemen-Sarti involution}
\label{subsubsec_Weierstrass model}

In this subsection, we study $H\oplus N$-polarized K3 surfaces. 
Their geometric properties, including the types of singular fibers of a Jacobian elliptic fibration and the van Geemen-Sarti involution associated to the alternate fibration, can be understood using different birational models for the K3 surface such as a suitable Weierstrass model and double quadrics model.

Let $\calX$ be a K3 surface with a $H\oplus N$-polarization, and let $\pi_\mathcal{X} \colon  \mathcal{X} \rightarrow \mathbb{P}^1$ be the alternate fibration. A Weierstrass model for $\pi_\calX$, which --by a slight abuse of notation-- we will also denote by $\calX$, is induced by the polarization with fibers in $\mathbb{P}^2 = \mathbb{P}(X, Y, Z)$ varying over $\mathbb{P}^1=\mathbb{P}(u, v)$. It is given by the equation
\beq
\label{eqn:vgs}
  \mathcal{X}\colon \quad Y^2 Z  = X \Big( X^2 + a(u, v) \, X Z + b(u, v) \, Z^2  \Big) \,,
\eeq
where $a$ and $b$ are homogeneous polynomials of degree four and eight, respectively.  The alternate fibration admits a section $\sigma\colon [X: Y: Z] = [0:1:0]$ and a 2-torsion section $[X: Y: Z] = [0:0:1]$, and it has the discriminant
\beq
\label{eqn:Delta_X}
\Delta_\mathrm{alt} =  b(u, v)^2 \, \big(a(u, v)^2-4  b(u, v)\big) \,.
\eeq
We also assume Equation~(\ref{eqn:vgs}) satisfies a \emph{minimality condition}:

\begin{lemma}
\label{lem:polynomials}
Assume that $b \neq 0$, $b \neq a^2/4$. The minimal resolution of Equation~(\ref{eqn:vgs}) is a K3 surface if and only if there is no polynomial $c \in \mathbb{C}[u, v]$ of $\deg{c} >0$ such that $c^2$ divides $a$ and $c^4$ divides $b$.
\end{lemma}
\begin{proof}
For $b=0$ or $b=a^2/4$, Equation~(\ref{eqn:vgs}) becomes $Y^2 Z = X^2 (X + a Z)$ or $Y^2Z=X (X + a Z/2)^2$ respectively. Otherwise, Equation~(\ref{eqn:vgs}) can be brought into the Weierstrass normal form
\beq
\label{eqn:WEq}
 Y^2 Z = X^3 + f(u, v) \, XZ + g(u, v) Z^3 .
\eeq
with coefficients
\beq
 f(u, v) =\frac{ 3 b(u, v) - a(u, v)^2}{3} \,, \qquad g(u, v) =  \frac{a(u, v) (2 a(u, v)^2 - 9 b(u, v) )}{27} \,, 
\eeq 
and discriminant $\Delta_\mathrm{alt}=4 f^3 + 27 g^2 = b^2 (a^2 - 4 b)$. One can check that Equation~(\ref{eqn:WEq}) has a point of vanishing orders $4, 6$ and $12$ with respect to $f, g$ and $\Delta_\mathrm{alt}$ respectively if and only if there is a polynomial $c$ so that $c^2$ divides $a$ and $c^4$ divides $b$.
\end{proof}

From Equation~(\ref{eqn:vgs}), one can check that a very general $H\oplus N$-polarized K3 surface $\mathcal{X}$ has 8 fibers of type $I_1$ over the zeroes of $a^2 - 4 b=0$ and 8 fibers of type $I_2$ over the zeroes of $b=0$ under the alternate fibration $\pi_\calX$.  The Mordell-Weil group $\operatorname{MW}(\mathcal{X}, \pi_\mathcal{X})$ is easily shown to be $ \mathbb{Z}/2\mathbb{Z}$.  
Let us  consider the translation by 2-torsion whose fiberwise action is given by
\beq
\label{eqn:VGS_involution}
 \imath_\mathcal{X}\colon \quad \Big[ X: Y : Z \Big] \mapsto \Big[ b(u, v) \, X Z \  : - b(u, v) \, Y Z \ : \ X^2 \Big] 
\eeq 
for $[X: Y: Z] \not= [0:1:0], [0:0:1]$, and by swapping $[0:1:0] \leftrightarrow [0:0:1]$. 
This is easily seen to be a Nikulin involution as it leaves the holomorphic 2-form invariant. 
The involution $\imath_\mathcal{X}$ is called a \emph{van Geemen-Sarti involution}. 

\begin{remark}
    If one factors  $\mathcal{X}$ by the action of $\imath_{\mathcal{X}}$ and then resolves the eight occurring $A_1$-type singularities, then one can obtain a new $H \oplus N$-polarized K3 surface $\mathcal{Y}$ with Weierstrass model
    \beq
\label{eqn:vgs_dual}
 \mathcal{Y}\colon \quad Y^2 Z = X \Big( X^2 - 2 a(u, v) \, X Z+ \big(a(u, v)^2 - 4 b(u, v)\big) \, Z^2 \Big) \,.
\eeq
    In fact, if one repeats the same  construction on $\calY$, then the original K3 surface $\mathcal{X}$ is recovered. 
    The two surfaces $\mathcal{X}$ and $\mathcal{Y}$ are therefore  referred as \emph{van Geemen-Sarti duals} (see \cite{Clingher:2020baq}).
\end{remark}

\begin{remark}
    One can construct a different birational model for $\calX$ as follows.
    For any $n \in N_0$, we consider the Hirzebruch surface $\FF_n$ as the GIT quotient
\beq
\label{eqn:HS}
 \mathbb{F}_n   \ \cong  \ \big( \mathbb{C}^2 - \{ 0 \} \big) \times \big( \mathbb{C}^2 - \{ 0 \} \big)  / \big( \mathbb{C}^\times \times \mathbb{C}^\times \big)  \,,
\eeq 
where the action of $\mathbb{C}^\times \times \mathbb{C}^\times$ on the coordinates $(u, v)$ and $(s, t)$ is given by
\beq
 (\lambda, \, \mu) \cdot \big( u, \, v, \, s, \, t \big) \ = \ \big( \lambda u, \, \lambda v, \, \mu s, \, \lambda^{-n} \mu t \big)\,.
\eeq
In this way, the branch locus in Equation~(\ref{eqn:vgs}) is contained in $\mathbb{F}_4$.  Suitable blowups transform Equation~(\ref{eqn:vgs}) into a double cover of $\mathbb{F}_0=\mathbb{P}(s, t) \times\mathbb{P}(u, v)$ branched along a curve of bi-degree $(4,4)$, i.e., along a section in the line bundle $\mathcal{O}_{\mathbb{F}_0}(4,4)$.  Such a cover is known as  \emph{double quadric surface} and the two rulings of the quadric $\mathbb{F}_0$, induced by the projections $\pi_i:~\mathbb{F}_0~\to~\mathbb{P}^1$ for $i=1,2$, give the alternate fibration and a standard fibration, respectively. For more details, see \cite{MR4544843}.
\end{remark}
%
%
\section{K3 surfaces with a different rank 10 polarization}\label{sec_Z family}

In this section, we study two equivalent birational models, namely the double sextics model and the projective quartic model, of a certain very general K3 surface $\calZ$ of Picard rank $10$.  We will prove in Proposition \ref{prop:fibration_rank10} that the associated Jacobian surface $J(\calZ)$ of $\calZ$ has an $H\oplus N$-polarization. We will then specialize the constructed model to give birational models for the K3 surfaces polarized by the lattices of rank $11$, $12$ and $13$ in Table \ref{tab:extension}. Note that some examples of equations relating the elliptic fibrations of K3 surfaces with 2-elementary N\'eron-Severi lattices and double sextics or quartic hypersurfaces have also been provided in \cites{MR4130832, MR3882710}.

\subsection{Double Sextics model}\label{subsec_double sextics Z}

 K3 surfaces that are the minimal resolution of double covers of $\mathbb{P}^2=\mathbb{P}(u, v, w)$ branched along plane sextic curves are called \emph{double sextic surfaces}, or double sextics for short.  The Picard number of a double sextic is in general at least one, and it increases if there exist curves in special position.    For example, a sextic might posses tritangents or contact conics.  
 In \cite{MR805337} this approach was used to construct a number of examples for sextics defining K3 surfaces with the maximal Picard number 20.

In this paper our focus is on K3 surfaces of lower Picard rank, and we consider double covers branched over the strict transform of a (reducible) sextic, given as the union of a nodal quartic $\mathcal{N}$ and a conic $\mathcal{C}$. 
Such a double sextic can be written in the form
\beq
\label{eqn:double-sextics_general}
   \mathcal{S}\colon \quad  \tilde{y}^2 =   C(u, v, w)  \cdot Q(u, v, w)\,,
\eeq
where $C$ and $Q$ are homogeneous polynomials of degree 2 and 4 respectively, such that $\mathcal{C}= \mathrm{V}(C)$ and $\mathcal{N}= \mathrm{V}(Q)$ and $\tilde{y}$ has weight 3.

The double sextics to be studied in this section are of the form
\beq
\label{eqn:double_sextic10}
   \mathcal{S}\colon \quad  \tilde{y}^2 =   \Big( j_0  u^2 + v w +  h_0 w^2 \Big)  \Big(c_2(u, v) \, w^2 + e_3(u, v) \, w + d_4(u, v) \Big),
\eeq
where $c_2, d_4, e_3$ are homogeneous polynomials of the degree indicated by the subscript and $h_0, j_0 \in \mathbb{C}$.
Note that in the branch locus of $\calS$, the irreducible component $\mathcal{N}=~\mathrm{V}(Q)$ has the node $\mathrm{n}\colon [u:v:w]=[0:0:1]$, and  $c_2(u, v)=0$ is the equation of the two tangents at $\mathrm{n}$.
The other irreducible component $\mathcal{C}=\mathrm{V}(C)$ is obtained from a general conic $h_0 w^2 + k_1(u, v) w + j_2(u, v)$ by coordinate shifts of the form $w \mapsto w + \rho_1 u + \rho_2 v$ and $u \mapsto u + \rho_0 v$ which keep the node $\mathrm{n}$ fixed.
Moreover, for $h_0 \neq 0$ the conic is not coincident with the node. 

The following shows that the minimal resolution $\calZ$ of $\calS$ is a K3 surface whose Jacobian surface is $H\oplus N$-polarized.
\begin{proposition}
\label{prop:fibration_rank10}
Let  $c_2, e_3, d_4$ be general polynomials  and $h_0, j_0 \in \mathbb{C}$ with $h_0 j_0 \neq 0$ and $\mathcal{S}$  the double sextic in Equation~(\ref{eqn:double_sextic10}). The pencil of lines through the node $\mathrm{n} \in \mathcal{N}$ induces an elliptic fibration without section.  The relative Jacobian fibration is Equation~(\ref{eqn:vgs}) with
\beq
\label{eqn:params_rank10}
 a  =   v e_3 -2 j_0 u^2 c_2 - 2 h_0 d_4 \,, \qquad b  = \frac{1}{4} a^2 -  \frac{1}{4} \big(e_3^2 - 4 c_2 d_4\big) \big(v^2 - 4 h_0 j_0 u^2\big) \,,
\eeq
and has singular fibers $8 I_2 + 8 I_1$ and Mordell-Weil group $\mathbb{Z}/2\mathbb{Z}$.
\end{proposition}
\begin{proof}
The pencil of lines through $\mathrm{n}$ induces an elliptic fibration given by the projection onto $\mathbb{P}(u, v)$ in Equation~(\ref{eqn:double_sextic10}). We then write Equation~(\ref{eqn:double_sextic10}) in the form
\beq
  \tilde{y}^2 = q_0 \, w^4 + q_1(u, v) \, w^3  +  q_2(u, v) \, w^2 + q_3(u, v) \, w + q_4(u, v) \, ,
\eeq
with $q_0=c_2 h_0, q_1=c_2 k_1 + e_3 h_0$, etc. The relative Jacobian of Equation~(\ref{eqn:double_sextic10}) is
\beq
\label{eqn:RJFIB}
  \eta^2 = \xi^3 +  f(u, v)\,  \xi + g(u, v) \,,
\eeq
where $f, g$ are given by 
\beq
\label{eqn:hermite}
 f = q_1 q_3 - 4 q_0 q_4  - \frac{1}{3} q_2^2 \,, \qquad g = q_0 q_3^2 + q_1^2 q_4 -\frac{8}{3} q_0 q_2 q_4 - \frac{1}{3} q_1 q_2 q_3 + \frac{2}{27} q_2^3 \,,
\eeq
and $(\xi,\eta) \in \mathbb{C}^2$ are affine coordinates on the elliptic fibers. One checks that Equation~(\ref{eqn:RJFIB}) can be brought into the form of Equation~(\ref{eqn:vgs_dual}) with $a, b$ given by Equation~(\ref{eqn:params_rank10}). 
\end{proof}

\subsection{Projective quartic model}\label{subsec_quartic Z}

Given a double sextics model $\calS$ as in Equation ~(\ref{eqn:double-sextics_general}), one can derive a projective quartic hypersurface $\mathcal{K} \subset \mathbb{P}^3 = \mathbb{P}(u, v, w, y)$ given by
\beq
\label{eqn:quartic_general}
   \mathcal{K}\colon \quad  0 =  - y^2 C(u, v, w) +  Q(u, v, w)\,,
\eeq
where $C$ and $Q$ are the same polynomials as in Equation~(\ref{eqn:double-sextics_general}).  
In this way, for every quartic surface $\mathcal{K}$ given by Equation~(\ref{eqn:quartic_general}) we find an associated double sextic $\mathcal{S}$ given by Equation~(\ref{eqn:double-sextics_general}) and vice versa. 

In this section, we specifically consider the family of quartic surfaces
\beq
\label{eqn:quartic_general2_10}
   \mathcal{K}\colon  \quad  0 = -  \big(j_0 u^2 + v w + h_0 w^2\big) y^2 +  c_2(u, v) \, w^2 + e_3(u, v) \, w + d_4(u, v) \,,
\eeq 
where $c_2, d_4, e_3$ are general homogeneous polynomials of the degree indicated by the subscript and $h_0, j_0 \in \mathbb{C}$. It has the associated double sextic $\mathcal{S}$ in Equation~(\ref{eqn:double_sextic10}). 
The quartic hypersurface $\calK$ is another birational model of the K3 surface $\calZ$:

\begin{proposition}
\label{prop:regular}
The minimal resolution of $\mathcal{K}$ in Equation~(\ref{eqn:quartic_general2_10}) is a smooth K3 surface if and only if Equation~(\ref{eqn:vgs}) with $a, b$ in~(\ref{eqn:params_rank10}) defines a minimal Weierstrass equation. In particular, $\mathcal{K}$ has exactly 2 singularities at points $[u: v: w : y]$, given by
\beq
\label{eqn:singular_points}
   \mathrm{p}_1 = [0: 0: 0: 1] \,, \qquad   \mathrm{p}_2 = [0 : 0: 1: 0] \,,
\eeq  
which are rational double points.
\end{proposition}
\begin{proof}
The double sextic in Equation~(\ref{eqn:double-sextics_general}) is birational to the quartic projective hypersurface in Equation~(\ref{eqn:quartic_general}) in $\mathbb{P}^3 = \mathbb{P}(u, v, w, y)$. 
This can be seen by setting $\tilde{y} = C(u, v, w) \, y$ in Equation~(\ref{eqn:double-sextics_general}).  
It follows from Proposition~\ref{prop:fibration_rank10} that the pencil of lines through the node of the quartic induces an elliptic fibration  on this double sextic whose relative Jacobian fibration is Equation~(\ref{eqn:vgs}).
Since the degree of Equation~(\ref{eqn:quartic_general}) is four, the minimal resolution of $\mathcal{K}$ is a smooth K3 surface if and only if Equation~(\ref{eqn:vgs}) is a minimal Weierstrass equation. 
Hence, the singularities of $\mathcal{K}$ must be rational double points. 
One checks that the points in Equation~(\ref{eqn:singular_points}) are the only singularities.
\end{proof}

\begin{remark}
    Given the existence of the singular point $\mathrm{p}_1$ on $\calK$ where the $y$ coordinate does not vanish, we may also see how the double sextic model $\calS$ naturally arises.
    By considering the lines through $p_1$ in $\PP^3$, we have a projection 
    \[\PP^3 - \{\mathrm{p}_1\} \longmapsto \PP^2.\]
    Restricting to $\calK-\mathrm{p}_1$, one checks that the image is the set of lines on $\calK$ tangent to $\mathrm{p}_1$, which is a sextic in $\PP^2$. Moreover, one can check that this map is a double cover.
\end{remark}
%
%
\section{Construction of elliptic fibrations of families of Picard rank 11, 12 and 13}\label{sec_fams}

In this section, we consider certain special configurations for the curves $\calC$ and $\calN$ in the branch locus of the double sextics $\calS$, which in turn correspond to specializations of Equation (\ref{eqn:double_sextic10}). Due to Proposition \ref{prop:regular}, this also specializes the corresponding projective quartic models. Particular normal forms for the specialized projective quartic hypersurfaces will also be established; they will be denoted by $\mathcal{Q} \subset \mathbb{P}^3=\mathbb{P}(\mathbf{X}, \mathbf{Y}, \mathbf{Z}, \mathbf{W})$. The normal forms turn out to be generalizations of the famous \emph{Inose quartic} \cite{MR578868}, given by
\beq
\label{eqn:inose_quartic}
 0= \  \mathbf{Y}^2 \mathbf{Z} \mathbf{W}-4 \mathbf{X}^3 \mathbf{Z}+3 \alpha \mathbf{X} \mathbf{Z} \mathbf{W}^2+ \beta \mathbf{Z} \mathbf{W}^3 
-   \frac{1}{2} \big( \mathbf{Z}^2  \mathbf{W}^2 +  \mathbf{W}^4 \big) .
\eeq

We will prove that the minimal resolution of each surface $\calS, \calK$ or $\calQ$ to be discussed is a K3 surface polarized by the lattices of rank $11, 12$ or $13$ in Table~\ref{tab:extension}.
From these models, we also compute explicit equations of pencils which induce the Jacobian elliptic fibrations supported on the families.  To do so we generalize the methods that were used in \cite{MR4544843} and \cite{Clingher:2020baq}, where the same was done for the families of Picard rank $14, 15$ and $16$ in Table~\ref{tab:extension}.

\subsection{Family of Picard rank 11}\label{subsec: fibrarion 11}

We first consider the case when the conic $\calC$ in Equation~(\ref{eqn:double_sextic10}) splits into $2$ lines. 
After a suitable shift in the coordinates $v$ and $ w$, the double sextic becomes
\beq
\label{eqn:double_sextic11}
   \mathcal{S}\colon \quad  \tilde{y}^2 =  w  \Big( v + h_0 w  \Big)   \Big(c_2(u, v) \, w^2 + e_3(u, v) \, w + d_4(u, v) \Big) \,.
\eeq
The branch locus  in Equation~(\ref{eqn:double_sextic11}) has three irreducible components: the quartic curve $\mathcal{N}=\mathrm{V}(c_2(u, v) \, w^2 + e_3(u, v) \, w + d_4(u, v))$ with the node $\mathrm{n} = [0: 0: 1]$,
and the lines $\ell_1 = \mathrm{V}(v + h_0 w), \ell_2 = \mathrm{V} (w)$ which are not coincident with $\mathrm{n}$ (for $h_0 \neq 0$) and satisfy $\ell_1 \cap \ell_2 \cap \mathcal{N} = \emptyset$. 

The following gives the geometric constructions of the elliptic fibrations induced by the double sextics $\calS$. 
\begin{proposition}
\label{prop:fibration_rank11}
Let  $c_2, e_3, d_4$ be  general polynomials  and $h_0 \in \mathbb{C}^\times$ and $\mathcal{S}$  the double sextic in Equation~(\ref{eqn:double_sextic11}).  The following holds:
\begin{enumerate}[label=(\roman*)]
 \item the pencil of lines through the node $\mathrm{n} \in \mathcal{N}$ induces the Jacobian elliptic fibration~(\ref{eqn:vgs}) with
\beq
\label{eqn:params_rank11} 
  a = v e_3 - 2 \, h_0 d_4, \qquad b=\frac{1}{4} a^2 - \frac{1}{4} v^2 \big(e_3^2 - 4 c_2 \, d_4 \big),
\eeq
and has singular fibers $9 I_2 + 6 I_1$ and Mordell-Weil group $\mathbb{Z}/2\mathbb{Z}$,
 \item a pencil of lines through a point  in $(\ell_1 \cap \mathcal{N}) \cup (\ell_2 \cap \mathcal{N})$ induces a Jacobian elliptic fibration with singular fibers $I_0^* + 5 I_2 + 8 I_1$ and a trivial Mordell-Weil group,
 \item the pencil of lines through $\ell_1 \cap \ell_2$ induces an elliptic fibration without section and singular fibers $2 I_0^* +  I_2 + 10 I_1$.
\end{enumerate}
\end{proposition}
\begin{proof}
For (i) we follow the proof of Proposition~\ref{prop:fibration_rank10}; it is obvious that the pencil  of lines now induces an elliptic fibration with two sections.  
We recover the normal form for the alternate fibration with singular fibers $9 I_2 + 6 I_1$.  
For (ii) we construct a pencil through a point in $\ell_2 \cap \mathcal{N}$, the case $\ell_1 \cap \mathcal{N}$ will then be analogous. 
To do so, we set $d_4(u, v) = (v_0 u - u_0 v) d_3(u, v)$ where $d_3$ is a polynomial of degree 3 with $d_3(u_0, v_0) \neq 0$.  
One point in $\mathcal{N} \cap \ell_2$  is $[u_0: v_0:  0]$. 
A pencil of lines through this intersection point is
\beqn
 0 = s \big( v_0 u - u_0 v \big) + t w
\eeqn 
with $[s: t ] \in \mathbb{P}^1$. 
Upon eliminating $u$  in Equation~(\ref{eqn:double_sextic11}) we obtain a Jacobian elliptic fibration over $\mathbb{P}(s, t)$ with singular fibers $I_0^* + 5 I_2 + 8 I_1$. 
For (iii) we use the pencil $0=s v - t w$ and proceed as in (ii). However, we only obtain a genus-one fibration over $\mathbb{P}(s, t)$. 
We compute the relative Jacobian elliptic fibration as in (i), confirming the types of the singular fibers. 
\end{proof}

We can verify that Equation~(\ref{eqn:double_sextic11}) parametrises a birational model for every very general $H\oplus D_4(-1) \oplus A_1(-1)^{\oplus 5}$-polarized K3 surfaces:
\begin{proposition}
\label{prop:rank11}
A very general $H\oplus D_4(-1) \oplus A_1(-1)^{\oplus 5}$-polarized K3 surface is birational to a double sextic $\mathcal{S}$ that is branched on a uninodal quartic $\mathcal{N}$ and 2 lines $\ell_1, \ell_2$ not coincident with the node and $\ell_1 \cap \ell_2 \cap \mathcal{N} = \emptyset$. Conversely, every such double sextic is birational to an $H\oplus D_4(-1) \oplus A_1(-1)^{\oplus 5}$-polarized K3 surface.
\end{proposition}
\begin{proof}
Due to Remark~\ref{rmk_2.3}, a very general $H\oplus D_4(-1) \oplus A_1(-1)^{\oplus 5}$-polarized K3 surface admits a unique Jacobian elliptic fibration~(\ref{eqn:vgs}) with singular fiber $9 I_2 + 6 I_1$. Then, after a suitable change of coordinates, one has $a(u, v)^2/4 - b(u, v) = v^2 b'(u, v)$ where $b'$ is a homogeneous polynomial of degree 6. Given the polynomials $a, b'$, we choose a factorization of $b(u, v) = a(u, v)^2/4 - v^2 b'(u, v)= d_4(u, v) d_4'(u, v)$ into homogeneous polynomials of degree 4 with $d_4(1,0) \not = 0$. We can then find $h_0 \in \mathbb{C^\times}$ so that $e_3 =(a + 2 h_0 d_4)/v$  and $c_2 =( d'_4(u, v)  + h_0^2 d_4(u, v) +h_0 a(u, v))/v^2$ are polynomials of degree 3 and 2 respectively. 

In fact, if we write $a = \sum_{n=0}^4 \alpha_n u^{4-n} v^{n}$, $d_4 = \sum_{n=0}^4  \delta_n u^{4-n} v^{n}$ (with $\delta_0 \not =0$), and $d'_4 = \sum_{n=0}^4  \delta'_n u^{4-n} v^{n}$, and if we have
\beq
 \delta'_0 = \frac{\alpha_0}{4 \delta_0} , \ \delta'_1 = \frac{\alpha_0 (2\alpha_1 \delta_0 - \alpha_0 \delta_1)}{4 \delta_0^2} , \  h_0 = - \frac{\alpha_0}{2 \delta_0},
\eeq 
then $h_0, c_2, e_3, d_4$ are a solution of Equation~(\ref{eqn:params_rank11}) for a given pair $a, b'$ of polynomials. 
This determine a double sextic $\mathcal{S}$ via the birational transformation $X = h_0 d_4 + d_4 v/w$.

The other direction follows from Proposition~\ref{prop:fibration_rank11}.
\end{proof}

Now following the same argument as in Section \ref{subsec_quartic Z}, we may construct a family of quartic surfaces
\beq
\label{eqn:quartic_general2_11}
   \mathcal{K}\colon  \quad  0 = -  \big(v - \rho w\big) w y^2 +  c_2(u, v) \, w^2 + e_3(u, v) \, w + d_4(u, v) \,
\eeq 
from Equation (\ref{eqn:double_sextic11}), assuming the conditions of Proposition~\ref{prop:regular}. In order to give explicit equations for the pencils of lines in Proposition~\ref{prop:fibration_rank11} which induce elliptic fibrations of the K3 surface, we introduce a normal form $\calQ$ for the projective quartic hypersurface $\calK$.
It is a multi-parameter generalizations of the Inose quartic in Equation~(\ref{eqn:inose_quartic}).  For Picard number 11, $\mathcal{Q}$ is given by the equation
\beq
\label{mainquartic11}
 \mathcal{Q}\colon \quad 0  =  2 \mathbf{Y}^2 \mathbf{Z} \big(\mathbf{W}-\rho \mathbf{Z}\big) - c_2\big(2 \mathbf{X} , \mathbf{W}\big)  \, \mathbf{Z}^2  - e_3\big(2 \mathbf{X} , \mathbf{W}\big) \,  \mathbf{Z} 
 - d_4\big(2 \mathbf{X} , \mathbf{W}\big) .
\eeq

One checks the following:
\begin{lemma}
\label{lem:relation}
Assume that the conditions of Proposition~\ref{prop:regular} are satisfied.  The surface $\mathcal{K}$ in Equation~(\ref{eqn:quartic_general2_11}) is isomorphic to the surface $\mathcal{Q}$ in Equation~(\ref{mainquartic11}), by
\beq
 \Big[ u : v : w : y \Big] \ \longmapsto \  \Big[ 2 \mathbf{X} : \mathbf{W} : \mathbf{Z} : \sqrt{2} \mathbf{Y} \Big] \,.
\eeq 
\end{lemma}

\begin{remark}\label{rmk: K, Q birational}
    By substituting $h_0 = -\rho$ and the same choices of $c_2, e_3$ and $d_4$ made in Proposition~\ref{prop:rank11} into Equation~(\ref{eqn:quartic_general2_11}), it follows that any $H\oplus D_4(-1) \oplus A_1(-1)^{\oplus 5}$-polarized K3 surface is birational to a quartic surface $\calK$. 
    Furthermore, Lemma~\ref{lem:relation} implies that $\calQ$ is a birational model: this will be shown explicitly in the proof of Theorem~\ref{thm_11}.
\end{remark}

\begin{remark}\label{rmk:sing 11}
    In \cite{MR2935386}, the authors obtained a configuration of rational curves on a very general $H\oplus E_8(-1) \oplus E_8(-1)$-polarized K3 surface by resolving the singularities on its projective quartic model $\calQ$. 
    Here, one checks by explicit computation that on a general surface $\mathcal{Q}$, there are two singularities at points $[ \mathbf{W} : \mathbf{X} : \mathbf{Y} :  \mathbf{Z} ]$ given by
\beq
 \mathrm{P}_1 = [0: 0: 1: 0]  \,, \qquad  \mathrm{P}_2 = [0: 0: 0: 1] \,.
\eeq  
Moreover, $\mathrm{P}_1$ and $\mathrm{P}_2$ are rational double points of types $A_3$ and $A_1$ respectively. In Section \ref{sec_dual graph 11}, the curves appearing from resolving $P_1$ and $P_2$ will be denoted by $a_1, a_2, a_3$ and $b_1$ respectively.
\end{remark}

Next, we introduce a set of complex coefficients $\{\gamma, \delta , \varepsilon, \zeta, \eta, \iota, \kappa, \lambda, \mu, \nu, \xi, \omicron\}$ such that the polynomials $c_2$ and $d_4$ in both Equations~(\ref{eqn:quartic_general2_11}) and (\ref{mainquartic11}) are given by 
\beq
\label{eqn:polys_2sextics11a}
\begin{split}
c_2(u, v) =  \big(\gamma u - \delta v\big)   \big(\varepsilon u - \zeta v\big)  \,, \qquad 
 d_4(u, v) =  \big(\eta u  - \iota v\big) \big(\kappa u -\lambda v\big)  \big(\mu u -\nu v\big) \big(\xi u -\omicron v\big).
\end{split} 
\eeq
With the set of coefficients, we can give explicit expressions for the vertices of the pencils of lines in Proposition~\ref{prop:fibration_rank11}. The intersection points of $\ell_2 = \mathrm{V} (w)$ and the nodal quartic $\mathcal{N}$ are given by the roots of the polynomial $d_4$. 
As for the vertices in $l_1 \cap \calN$, we need to introduce more coefficients. First note that for $\rho \neq 0$, the line $\ell_1 = \mathrm{V} (v - \rho w)$ is not coincident with the node $\mathrm{n} \in \mathcal{N}$. 
Moreover, we can assume $\gamma \varepsilon \eta \kappa \mu \xi \neq 0$ in the general case. 
To make the points $\ell_1\cap \mathcal{N}$ explicit, we set $w = (v+\tilde{w})/\rho$ in Equation~(\ref{eqn:quartic_general2_11}) and obtain an equivalent equation of the form
\beq
\label{eqn:quartic_general2_11b}
   \mathcal{K}\colon  \quad  0 = - \rho \big( v + \tilde{w}\big) \tilde{w} y^2 +  c_2(u, v) \, \tilde{w}^2 + e'_3(u, v) \, \tilde{w} + d'_4(u, v) \,,
\eeq 
where $\ell_1 = \mathrm{V} (\tilde{w})$ and
\beq
\label{eqn:transfo}
e'_3(u, v) =  \rho e_3(u, v) + 2  c_2(u, v) \, v \,, \qquad  d'_4(u, v)  = \rho^2 d_4(u, v) + \rho e_3(u, v) \, v + c_2(u, v) \, v^2.
\eeq
The roots of $d'_4$ determine the intersection points of $\ell_1\cap \mathcal{N}$.
So, by writing
\beq
\label{eqn:polys_2sextics11b}
\begin{split}
 d'_4(u, v) =  \big(\eta' u - \iota' v\big)   \big(\kappa' u -\lambda' v\big)  \big(\mu' u -\nu' v\big) \big(\xi' u -\omicron' v\big), 
\end{split} 
\eeq
we have the coordinates for the points in $\ell_1\cap \mathcal{N}$.

Note that the expanded coefficient set $(\gamma, \delta , \varepsilon, \zeta , \eta, \dots , \omicron, \eta', \dots , \omicron')$ also determines the polynomials $\rho, e_3, e'_3$ in Equations~(\ref{eqn:quartic_general2_11}),  (\ref{mainquartic11}) and (\ref{eqn:quartic_general2_11b}), using the equations
\beq
\label{eqn:e3polynomials}
\begin{split}
\rho^2 = \frac{\eta' \kappa'  \mu' \xi'}{\eta \kappa \mu \xi} ,&  \qquad \quad e_3(u, v)  = \frac{d'_4(u, v)  - \rho^2 d_4(u, v) -  c_2(u, v) \, v^2}{\rho v} ,  \\
e'_3(u, v) & = \frac{d'_4(u, v)  - \rho^2 d_4(u, v) +  c_2(u, v) \, v^2}{v} .
\end{split}
\eeq

\begin{remark}\label{rmk:X<=>K}
    We may also express the Weierstrass model $\calX$ (\ref{eqn:vgs}) in terms of these complex coefficients in Equation~(\ref{eqn:double_sextic11}).
    By substitution of Equation~(\ref{eqn:params_rank11}), the polynomials $a$ and $b$ in the Jacobian elliptic fibration~(\ref{eqn:vgs}) are given by
\beq
\label{eqn:ab_polynomials}
 a(u, v) =  \frac{d'_4(u, v)  + \rho^2 d_4(u, v) -  c_2(u, v) \, v^2}{\rho}  , \qquad b(u, v) = d_4(u, v) \, d'_4(u, v) \,.
\eeq
Conversely, the proof of Proposition~\ref{prop:rank11} shows how the Weierstrass model $\calX$ determines the quartic surface $\calK$.
\end{remark}

\begin{remark}
\label{rem:signs}
Different sign choices $\pm \rho$ in Equation~(\ref{eqn:e3polynomials}) (resulting in sign changes $\pm e_3$  and $\pm a$ in Equation~(\ref{eqn:ab_polynomials})) yield isomorphic surfaces in Equation~(\ref{eqn:quartic_general2_11}) (resp.~ in Equation~(\ref{eqn:vgs})) related by $[u: v: w: y] \mapsto [u: v: - w: i y]$ (resp.~$[X:Y:Z] \mapsto [-X: iY : Z]$).
\end{remark}

Let us resume the computation of explicit expressions for the pencils of lines in Proposition~\ref{prop:fibration_rank11}. For $\mathcal{Q}$ in Equation~(\ref{mainquartic11}), we now introduce the lines  $L_4$, $L_5$, $L_6$, $L_7$ on $\calQ$ such that each line passes through a different intersection point in $l_2 \cap \calN$, as well as the lines  $L'_4$, $L'_5$, $L'_6$, $L'_7$ on $\calQ$ such that each line passes through a different intersection point in $l_2 \cap \calN$.
To be specific, we define the lines $L_4$ and $L_4'$ by the equations
\beq\label{eqn: Li,Li'}
 L_4\colon  \quad  \mathbf{Z}=2 \eta \mathbf{X} - \iota \mathbf{W}=0 , \qquad
 L'_4\colon \quad \mathbf{W} - \rho \mathbf{Z}=2 \eta' \mathbf{X} - \iota' \mathbf{W}=0.
\eeq
From the equation for $L_4$, we define equations for $L_5$, $L_6$ and $L_7$ by replacing the parameters $(\eta, \iota)$ by $(\kappa, \lambda)$, $ (\mu, \nu)$, and $(\xi, \omicron)$ respectively. 
Similarly, from the equation of $L_4'$, equations for $L'_5$, $L'_6$, and $L'_7$ can be defined  by replacing the parameters $(\eta', \iota') $ by $(\kappa', \lambda')$, $(\mu', \nu')$, and $(\xi', \omicron')$ respectively.  For general parameters, the lines are distinct and concurrent, meeting at $\mathrm{P}_1$. Each line induces a pencil; they are denoted by $L_i(u, v)$ and $L'_i(u, v)$ for $i=4,5,6,7$. For example, we set $L_4(u, v)= u \mathbf{Z} - v (2 \eta \mathbf{X} - \iota \mathbf{Z})$.  
Combined with Proposition~\ref{prop:fibration_rank11}, we then have the following geometric constructions of Jacobian elliptic fibrations:
\begin{theorem}
\label{thm_11}
Assume that the conditions of Proposition~\ref{prop:regular} are satisfied and let $L=H \oplus D_4(-1) \oplus A_1(-1)^{\oplus 5}$.  The minimal resolution of $\mathcal{Q}$ in Equation~(\ref{mainquartic11}) is a K3 surface endowed with a canonical  $L$-polarization. Conversely, an $L$-polarized K3 surface has a birational projective model~(\ref{mainquartic11}).  In particular,  Jacobian elliptic fibrations of the type determined in Theorem~\ref{thm:families} are attained as follows:
\beqn
\begin{array}{c|c|c|c|l}
\# 	&  \text{singular fibers} 		& \operatorname{MW} 		& \text{reducible fibers} 										& \text{pencil} \\
\hline
&&&&\\[-0.9em]
1	& 9 I_2 + 6 I_1		&  \mathbb{Z}/2\mathbb{Z}	& {\color{magenta} \widetilde{A}_1}^{\oplus 9}  						& \text{family of quartic curves through $\mathrm{P}_1$ and $\mathrm{P}_2$,} \\ 
	&				&						&														& \text{intersection of $2 v \mathbf{X} + u \mathbf{W}=0$ and $\mathcal{Q}$}\\
\hline
&&&&\\[-0.9em]
2	& I_0^* + 5 I_2 + 8 I_1		&   \{ \mathbb{I} \} 			& {\color{blue} \widetilde{D}_4} + {\color{magenta} \widetilde{A}_1}^{\oplus 5} 	& \text{residual surface intersection of $L_i(u, v)=0$} \\ 
	&						&						&													& \text{or $L'_i(u, v)=0 \; (i=4,5,6,7)$ and $\mathcal{Q}$}
\end{array}
\eeqn
\end{theorem}

\begin{proof}
To obtain the fibration~(1) we make the substitution
\beq
\label{eqn:substitution_alt_11}
\begin{split}
 \mathbf{W} = 2 v^2 X \big(X + \rho \, d_4(u, v) Z\big), & \qquad
 \mathbf{X} =  u v X \big(X + \rho \, d_4(u, v) Z\big), \\
 \mathbf{Y}=  \sqrt{2} Y \big(X + \rho \, d_4(u, v) Z\big), & \qquad 
 \mathbf{Z} = 2 v^2 d_4(u, v) \, XZ ,
\end{split} 
\eeq
in Equation~(\ref{mainquartic11}). This determines the Jacobian elliptic fibration~(\ref{eqn:vgs}) where $a, b$ are given by Equation~(\ref{eqn:params_rank11}) with $h_0 = -\rho$. Fibration~(1) is induced by a family of quartic curves through $\mathrm{P}_1$ and $\mathrm{P}_2$, which is obtained as the intersection of the pencil of hyperplanes $2 v \mathbf{X} + u \mathbf{W}=0$ and $\mathcal{Q}$. In general, $\mathbf{X}=\mathbf{W}=0$ does not define a line on $\mathcal{Q}$ if $\rho \neq 0$.
\par The fibrations induced by the pencils $L_n(u, v)=0$ and $L'_n(u, v)=0$ for $n=4, 5, 6, 7$ are all found in the same way. For example, if one substitutes $L_4(u, v) =0$ into Equation~(\ref{eqn:params_rank11}), one obtains a genus-one fibration. The fibration has a section since $2 \eta \mathbf{X} - \iota \mathbf{W} =0$ is a root of the polynomial $d_4(2 \mathbf{X} , \mathbf{W})$. By bringing the equation into Weierstrass normal form one obtains fibration~(2). The fiber of type $I_0^*$ is located over $u=0$, one fiber of type $I_2$ over $v=0$, and the remaining fibers of type $I_2$ over
\beq
\label{eqn:base_points}
\frac{u}{v} = \frac{\rho (\eta \iota' -\iota  \eta' )}{\eta'} , \frac{\rho (\eta \lambda' - \iota \kappa' )}{\kappa'} ,  \frac{\rho (\eta \nu' - \iota \mu' )}{\mu'} ,  \frac{\rho (\eta \omicron' - \iota \xi' )}{\xi'} \,.
\eeq 
Similar results can be obtained for the pencils $L_i(u, v)$ for $i=5, 6, 7$. Using Lemma~\ref{lem:relation} and Equation~(\ref{eqn:quartic_general2_11b}) as an equivalent way of writing Equation~(\ref{eqn:quartic_general2_11}), it follows that the same results hold if we interchange the lines $L_n$ and $L'_n$ for $n= 4, 5, 6, 7$. 

The converse is given by Proposition~\ref{prop:rank11} and Remark~\ref{rmk: K, Q birational}.
Alternatively, upon solving Equations~(\ref{eqn:substitution_alt_11}) for $u, v, X, Y$ and plugging the result into Equation~(\ref{eqn:vgs}), the proper transform is a quartic surface $\mathcal{Q}$ which is given by an equation of the form~(\ref{mainquartic11}). 
 \end{proof}

\begin{remark}\label{rmk: nikulin 11}
    Assume that the conditions of Proposition~\ref{prop:regular} are satisfied. 
    The van Geemen-Sarti involution (\ref{eqn:VGS_involution}), when constructed for Fibration~(1) in Theorem~\ref{thm_11}, induces a Nikulin involution on the K3 surface, that is the minimal resolution of $\mathcal{Q}$ in Equation~(\ref{mainquartic11}). The Nikulin involution is induced by the projective automorphism  $\mathbb{P}^3\dashrightarrow  \mathbb{P}^3$ given by
\beq
\label{eqn:involution_quartic11}
\begin{split}
& \Big[ \mathbf{W} : \ \mathbf{X} : \ \mathbf{Y} :  \ \mathbf{Z} \Big] 
\ \mapsto \ \\
& \quad \Big[ \tilde{d}_4\left(2 \mathbf{X}, \mathbf{W} , \mathbf{Z} \right) \mathbf{W} : \tilde{d}_4\left(2 \mathbf{X}, \mathbf{W} , \mathbf{Z} \right)\mathbf{X} : -\tilde{d}_4\left(2 \mathbf{X}, \mathbf{W} , \mathbf{Z} \right) \mathbf{Y} :  d_4\left(2 \mathbf{X}, \mathbf{W}  \right) \left( \mathbf{W}  - \rho \mathbf{Z} \right) \Big]
\end{split}
\eeq
with $\tilde{d}_4(u, v, w) = \rho d_4(u, v) + \rho w e_3(u, v) + v w  c_2(u, v)$. Equation~(\ref{eqn:involution_quartic11}) is obtained by writing the action of the involution (\ref{eqn:VGS_involution}) in terms of the coordinates $\mathbf{W}, \mathbf{X}, \mathbf{Y}, \mathbf{Z}$ using Equations~(\ref{eqn:substitution_alt_11}).
\end{remark}
%
%
\subsection{Family of Picard rank 12}\label{subsec: fam 12}
We continue to specialize the double sextic model (\ref{eqn:double_sextic11}) to obtain a geometric construction of elliptic fibrations of a very general $H\oplus D_6(-1) \oplus A_1(-1)^{\oplus 4}$-polarized K3 surfaces. We will follow the structure and organization of Section~\ref{subsec: fibrarion 11}, skipping over similar proofs and computations.

We consider the case when one of the lines in Equation~(\ref{eqn:double_sextic11}) becomes coincident with the node by setting $h_0=0$, i.e., 
\beq
\label{eqn:double_sextic12}
   \mathcal{S}\colon \quad  \tilde{y}^2 =    v w    \Big(c_2(u, v) \, w^2 + e_3(u, v) \, w + d_4(u, v) \Big) \,.
\eeq
The branch locus  in Equation~(\ref{eqn:double_sextic12}) has three irreducible components: the quartic curve $\mathcal{N}=\mathrm{V}(Q)$ with the node $\mathrm{n}=[0:0:1]$ and the lines $\ell_1 = \mathrm{V} (v),  \ell_2 = \mathrm{V}(w)$ such that $\mathrm{n} \in \ell_1 \cap \mathcal{N}$ and $\ell_1 \cap \ell_2 \cap \mathcal{N} = \emptyset$. 
The following proposition describes the elliptic fibrations induced by $\calS$.
\begin{proposition}
\label{prop:fibration_rank12}
Let  $c_2, e_3, d_4$ general polynomials and $\mathcal{S}$  the double sextic in Equation~(\ref{eqn:double_sextic12}). The following holds:
\begin{enumerate}[label=(\roman*)]
 \item the pencil of lines through the node $\mathrm{n} \in \mathcal{N}$ induces the Jacobian elliptic fibration~(\ref{eqn:vgs}) with
 \beq
 \label{eqn:params_rank12}
 a  =  v \, e_3 \,, \qquad b  =v^2 c_2\, d_4 \,,
\eeq
and has singular fibers $I_0^* + 6 I_2 + 6 I_1$ and Mordell-Weil group $\mathbb{Z}/2\mathbb{Z}$,
 \item a pencil of lines through a point in $\ell_2 \cap \mathcal{N}$ induces a Jacobian elliptic fibration with singular fibers $2 I_0^* + 2 I_2 + 8 I_1$ and a trivial Mordell-Weil group, 
 \item a pencil of lines through a point in $\ell_1 \cap \mathcal{N} - \{ \mathrm{n} \}$ induces a Jacobian elliptic fibration with singular fibers  $I_2^* +  4 I_2 + 8 I_1$ and a trivial Mordell-Weil group,
 \item  the pencil of lines through $\ell_1 \cap \ell_2$ induces an elliptic fibration without section and singular fibers $I_2^* + I_0^* + 10 I_1$.
\end{enumerate}
\end{proposition}
\begin{proof}
The proof is analogous to the proof of Proposition~\ref{prop:fibration_rank11}. 
\end{proof}

It can be proved that Equation~(\ref{eqn:double_sextic12}) parametrises a birational model for every very general $H\oplus D_6(-1) \oplus A_1(-1)^{\oplus 4}$-polarized K3 surface

\begin{proposition}
\label{prop:rank12}
A very general $H\oplus D_6(-1) \oplus A_1(-1)^{\oplus 4}$-polarized K3 surface is birational to a double sextic $\mathcal{S}$ that is branched on a uninodal quartic $\mathcal{N}$ (with node $\mathrm{n}$) and 2 lines $\ell_1, \ell_2$ so that $\mathrm{n} \in \ell_1 \cap \mathcal{N}$ and  $\ell_1 \cap \ell_2 \cap \mathcal{N} = \emptyset$. Conversely, every such double sextic is birational to an $H\oplus D_6(-1) \oplus A_1(-1)^{\oplus 4}$-polarized K3 surface.
\end{proposition}
\begin{proof}
Due to Remark \ref{rmk_2.3}, a very general $H\oplus D_6(-1) \oplus A_1(-1)^{\oplus 4}$-polarized K3 surface admits a unique Jacobian elliptic fibration~(\ref{eqn:vgs}) with a fiber of type $I_0^*$.  We have $a(u, v) = v e_3(u, v)$, and we may group the base points of the six fibers of type $I_2$ over $b=0$ into sets of 2 and 4 elements by writing $b(u, v)=v^2 c_2(u, v) d_4(u ,v)$ where $c_2, e_3, d_4$ are homogenous polynomials of degree 2, 3, and 4, respectively. We obtain the double sextic given by Equation~(\ref{eqn:double_sextic12}) by blowing up in Equation~(\ref{eqn:vgs}) via $[x:y] = [ v c_2(u, v) \tilde{x} : \tilde{y} ]$  with $[(u, v, x, z)] \in \mathbb{F}_4$ and $[(u, v, \tilde{x}, \tilde{z})] \in \mathbb{F}_1$ (using the notation of Equation~(\ref{eqn:HS})) and identifying $\mathbb{F}_1 \cong \mathbb{P}(u, v, w)$ via $w = \tilde{x}/\tilde{z}$ for $\tilde{z} \neq 0$. 
One checks that $\ell_1 = \mathrm{V}(v)$ is coincident with the node whereas $\ell_2 = \mathrm{V}(w)$ is not.  The other direction follows from Proposition~\ref{prop:fibration_rank12}.
\end{proof}

From the family $\calS$, we construct a family of quartic surfaces
\beq
\label{eqn:quartic_general2_12}
   \mathcal{K}\colon  \quad  0 = - v w   y^2 +  c_2(u, v) \, w^2 + e_3(u, v) \, w + d_4(u, v) \,.
\eeq 
Again, we assume the conditions of Proposition~\ref{prop:regular}.
It can be brought to the form of a generalized Inose quartic $\calQ$ as in Equation~(\ref{eqn:inose_quartic}) with $\rho = 0$.
By Lemma~\ref{lem:relation}, $\calK$ and $\calQ$ are isomorphic quartic surfaces.
Combined with Propositions~\ref{prop:regular} and \ref{prop:rank12}, we can see that both $\calK$ and $\calQ$ are birational models of a very general $H\oplus D_6(-1) \oplus A_1(-1)^{\oplus 4}$-polarized K3 surface.

By introducing a set of complex coefficients for Equation~\ref{eqn:quartic_general2_12}, we may obtain explicit expressions for the pencils which induce the Jacobian elliptic fibrations for a very general $H\oplus D_6(-1) \oplus A_1(-1)^{\oplus 4}$-polarized K3 surface.
By assumption, $\calK$ is very general so we may assume that the leading coefficients of $c_2, e_3$ and $d_4$ in Equation~(\ref{eqn:quartic_general2_12}) do not vanish.
In particular, we can write
\beq
\label{eqn:polys_2sextics}
\begin{split}
c_2(u, v) =  \big(\gamma u - \delta v\big)  \big(\varepsilon u - \zeta v\big)  , \qquad  \quad
 e_3(u, v) = u^3 - 3 \alpha u v^2 - 2 \beta v^3 , \\
 d_4(u, v) =  \big(\eta u  - \iota v\big)  \big(\kappa u -\lambda v\big)  \big(\mu u -\nu v\big) \big(\xi u -\omicron v\big) , \qquad
\end{split} 
\eeq
where the coefficients $(\alpha, \beta, \gamma, \delta , \varepsilon, \zeta, \eta, \iota, \kappa, \lambda, \mu, \nu, \xi, \omicron) \in \mathbb{C}^{14}$ satisfies $\gamma \varepsilon \eta \kappa \mu \xi \neq 0$. 

\begin{remark}\label{rmk:X<=>K 12}
    Similar to Remark~\ref{rmk:X<=>K}, one can obtain the polynomials $a$ and $b$ in the Weierstrass model $\calX$ (\ref{eqn:vgs}) from $\calK$ (\ref{eqn:quartic_general2_12}) using Equation~(\ref{eqn:params_rank12}).

    The converse also holds true, as in the proof of Proposition~\ref{prop:rank12}.
    Explicitly, if one substitutes Equation~(\ref{eqn:polys_2sextics}) into 
\beq\label{eqn:b=cd 12}
    a(u,v):= \sum_{i=0}^4 a_iv^iu^{4-i} = v e_3(u,v), \hspace{1em}
    b(u,v):= \sum_{j=0}^8 b_jv^ju^{8-j} = v^2 c_2(u,v)d_4(u,v),
\eeq
then 
the coefficients $a_i$ and $b_j$ are as follows:
\beq
\label{eqn:coeffs_set}
\begin{split}
 a_0=0, \quad a_1=1, \quad a_2=0, \quad a_3 =  -3 \alpha, \quad  a_4=- 2 \beta,\\
 b_0=b_1=0, \quad b_2 = \gamma \varepsilon \eta \kappa \mu \xi , \quad b_j =   (-1)^j  \gamma \varepsilon \eta \kappa \mu \xi  \cdot \sigma^{(6)}_{j-2}\left( \frac{\delta}{\gamma} , \dots , \frac{\omicron}{\xi} \right), \quad b_8 = \delta\zeta\iota\lambda\nu \omicron.
\end{split}
\eeq
Here, $\sigma^{(6)}_{k-2}$ for $k=2, \dots, 8$ are the elementary symmetric polynomials in 6 variables of degree $k-2$.  Therefore, polynomials $c_2, e_3$ and $d_4$ in Equation~(\ref{eqn:polys_2sextics}) can be expressed in terms of the $a_i$ and $b_j$. 
\end{remark}

\begin{remark}\label{rmk: sing 12}
Suppose $(\alpha, \beta, \gamma, \delta , \varepsilon, \zeta, \eta, \iota, \kappa,\lambda,\mu, \nu, \xi, \omicron) \in \mathbb{C}^{14}$ is the coefficient set considered in Remark~\ref{rmk:X<=>K 12}.
    Since \[b(u,v) = v^2 \big(\gamma u - \delta v\big) \big(\varepsilon u - \zeta v\big)  \big(\eta u  - \iota v\big)  \big(\kappa u -\lambda v\big)  \big(\mu u -\nu v\big) \big(\xi u -\omicron v\big),\]
    the vanishing order $N_{[1: -r]}(b) \in \{ 0, 1, \dots, 6\}$ of $b(u, v)$ at $[u: v]= [1: -r]$ for $r \neq 0$ equals the number of parameters $(u, v) \neq 0$ with
\beq
 [u: v] \ \in \ \Big\{  \big[\gamma, \delta \big],   \big[\varepsilon, \zeta \big],  \big[\eta, \iota \big],  \big[\kappa, \lambda \big],  \big[\mu,  \nu \big],  \big[\xi, \omicron\big] \Big\} \,.
\eeq
such that $[u: v] = [1: -r]$.
If we have $b(u, v) \neq 0, \frac{1}{4} a(u, v)^2$, and if  $\not \exists r \in \mathbb{C}$ with $ \ \alpha= r^2, \, \beta=r^3 ,\, N_{[1: -r]}(b)  \ge 4$, then it can be checked by explicit computation that the rational double points $\mathrm{P}_1$ and $\mathrm{P}_2$ in Remark~\ref{rmk:sing 11} are both of types $A_3$. 
\end{remark}

\begin{remark}\label{rmk:symmetries 12}
Let $\mathcal{K}$ and $\widetilde{\mathcal{K}}$ be two quartic surfaces satisfying the conditions of Remark~\ref{rmk: sing 12}.
By Proposition~\ref{prop:fibration_rank12}, their associated double sextics admit a Jacobian elliptic fibration~(\ref{eqn:vgs}) with polynomials $a, b$ and $\tilde{a}, \tilde{b}$ respectively. 
Upon bringing these polynomials into the standard form, it follows that the K3 surfaces obtained as minimal resolutions of $\calK$ and $\widetilde{\mathcal{K}}$ are isomorphic if and only if
\beq
\label{eqn:symmetry3}
 \Lambda^2 a(u, v) = \tilde{a}(u, \Lambda^2 v) \,, \qquad  \Lambda^4 b(u, v) = \tilde{b}(u, \Lambda^2 v) \,,
\eeq
for some $\Lambda \in \mathbb{C}^ \times$. In other words, the coefficients of $\tilde{a}, \tilde{b}$ are obtained from those of $a, b$ by the following:
\begin{itemize}
\item [(a)] pairwise interchanges of the elements in  
\beqn
\label{eqn:symmetry1}
 \Big\{ (\gamma, \delta), (\varepsilon, \zeta), (\eta, \iota), (\kappa, \lambda), (\mu,\nu), (\xi, \omicron)\Big\} ,
\eeqn
\item [(b)] rescalings of the coefficients according to
\beqn
\label{eqn:symmetry2}
\begin{split}
& (\alpha,\beta, \gamma, \delta, \varepsilon, \zeta, \eta,  \iota, \kappa, \lambda, \mu, \nu, \xi, \omicron) \ \mapsto \\ 
& \qquad (\Lambda^4 \alpha,  \Lambda^6 \beta,  \Lambda^{10} \gamma,  \Lambda^{12} \delta,  \Lambda^{-2} \varepsilon,  \zeta , \Lambda^{-2} \eta, \iota, \Lambda^{-2} \kappa, \lambda, \Lambda^{-2} \mu, \nu,  \Lambda^{-2} \xi, \omicron ),
\end{split}
\eeqn
for some $\Lambda \in \mathbb{C}^\times$.
\end{itemize}
\end{remark}

Recall that the points $l_2\cap \calN$ mentioned in Proposition~\ref{prop:fibration_rank12} have already been found in Section~\ref{subsec: fibrarion 11}.
Again, they give rise to the pencils of lines $L_i$ for $i = 4, \cdots, 7$ given in Equation~(\ref{eqn: Li,Li'}).
For $\mathcal{Q}$ now given by
\beq
\label{mainquartic}
\begin{split}
  \mathcal{Q}\colon \quad  0= \ 2 \mathbf{Y}^2 \mathbf{Z} \mathbf{W}& -8 \mathbf{X}^3 \mathbf{Z}+6 \alpha \mathbf{X} \mathbf{Z} \mathbf{W}^2+2\beta \mathbf{Z} \mathbf{W}^3 
-  \big(  2 \gamma \mathbf{X} - \delta \mathbf{W}  \big)  \big(  2 \varepsilon \mathbf{X} - \zeta \mathbf{W}  \big)    \mathbf{Z}^2 \\
 & -   \big(  2 \eta \mathbf{X} - \iota \mathbf{W}  \big) \big(  2 \kappa \mathbf{X} - \lambda \mathbf{W}  \big)  \big(  2 \mu \mathbf{X} - \nu \mathbf{W}  \big)  \big(  2 \xi \mathbf{X} - \omicron \mathbf{W}  \big),
\end{split}
\eeq
we find additional lines denoted by $L_1 $, $L^\pm_2$:
\beq\label{eqn:chi}
 L_1\colon  \quad \mathbf{W}=\mathbf{X}=0,  \qquad
 L^{\pm}_2\colon \quad  \mathbf{W}=(1 \pm \chi) \mathbf{X} + \gamma \varepsilon \mathbf{Z}=0, 
\eeq
where the new parameter $\chi$ satisfies  $4 \gamma \varepsilon \eta \kappa \mu \xi = 1 - \chi^2$.  The lines $L_1$, $L^{\pm}_2$, $L_4, L_5, L_6, L_7$ lie on the quartic surface $\mathcal{Q}$ in Equation~(\ref{mainquartic}). For general parameters, the lines are distinct and meet at $\mathrm{P}_1$. The line $L_1$ is also coincident with $\mathrm{P}_2$. 
Again, we denote by $L_1(u,v)$ and $L_2^\pm(u,v)$ the pencils of lines associated to the lines $L_1$ and $L_2^\pm$.
Note that the above pencils correspond to the pencils mentioned in Proposition~\ref{prop:fibration_rank12} with vertices being $n\in \calN$ and the points in $l_1 \cap \calN - \{n\}$ respectively.
We then have the following geometric construction of the Jacobian elliptic fibrations:
\begin{theorem}
\label{thm_12}
Assume that the conditions of Remark~\ref{rmk: sing 12} are satisfied, and let $L= H \oplus D_6(-1) \oplus A_1(-1)^{\oplus 4}$.  The minimal resolution of $\mathcal{Q}$ in Equation~(\ref{mainquartic}) is a K3 surface endowed with a canonical $L$-polarization. Conversely, a very general $L$-polarized K3 surface has a birational projective model~(\ref{mainquartic}). In particular,  Jacobian elliptic fibrations of the type determined in Theorem~\ref{thm:families} are attained as follows:
\beqn
\begin{array}{c|c|c|c|l}
\# 	&  \text{singular fibers} 		& \operatorname{MW} 		& \text{reducible fibers} 										& \text{pencil} \\
\hline
&&&&\\[-0.9em]
1	& I_0^* + 6 I_2 + 6 I_1		&  \mathbb{Z}/2\mathbb{Z}	& {\color{blue} \widetilde{D}_4 } + {\color{magenta} \widetilde{A}_1}^{\oplus 6}  		& \text{residual surface intersection} \\ 
	&						&						&														& \text{of $L_1(u, v)=0$ and $\mathcal{Q}$}\\
\hline
&&&&\\[-0.9em]
2	& 2 I_0^* + 2 I_2 + 8 I_1		&   \{ \mathbb{I} \} 			& {\color{blue} \widetilde{D}_4} + {\color{green} \widetilde{D}_4} + {\color{magenta} \widetilde{A}_1}^{\oplus 2} & \text{residual surface intersection} \\ 
	&						&						&														& \text{of $L_i(u, v)=0 \; (i=4,5,6,7)$ and $\mathcal{Q}$}\\
\hline
&&&&\\[-0.9em]
3	& I_2^* + 4  I_2 + 8 I_1		&   \{ \mathbb{I} \} 			& {\color{blue} \widetilde{D}_6}+ {\color{magenta} \widetilde{A}_1}^{\oplus 4} 		& \text{residual surface intersection} \\ 
	&						&						&														& \text{of $\mathit{L}^\pm_2(u, v)=0$ and $\mathcal{Q}$}
\end{array}
\eeqn
\end{theorem}

\begin{proof}
The proof is analogous to the proof of Theorem~\ref{thm_11}.  To obtain the fibration~(1) we make the substitution
\beq
\label{eqn:substitution_alt}
\mathbf{W} = 2 v^2 x, \quad  \mathbf{X}= u v x, \quad  \mathbf{Y}=  \sqrt{2} y , \quad \mathbf{Z} = 2 v^2 (\eta u-\iota v) (\kappa u- \lambda v)  (\mu u- \nu v)  (\xi u- \omicron v) z, 
\eeq
in Equation~(\ref{mainquartic}), compatible with the pencil $L_1(u, v)=0$. This determines the Jacobian elliptic fibration~(\ref{eqn:vgs}) where the polynomials $a, b$ are as in Remark~\ref{rmk:X<=>K 12}.
The other two fibrations are found in a similar way as in the proof of Theorem~\ref{thm_11}.
If one substitutes $L_n(u,v)=0$ for $n=4, 5, 6, 7$ (resp. $L_2^\pm(u,v)$) into Equation~(\ref{eqn:params_rank12}), then one obtains a genus-one fibration with section.
By bringing the equation into Weierstrass normal form one obtains Jacobian elliptic fibration~(2) (resp. (3)).
The corresponding equations for the other pencils are then generated through the symmetries in Remark~\ref{rmk:symmetries 12}.

The converse direction is given by Proposition~\ref{prop:rank12} and Remark~\ref{rmk: K, Q birational}.
Alternatively, upon solving Equations~(\ref{eqn:substitution_alt}) for $u, v, X, Y$ and plugging the result into Equation~(\ref{eqn:vgs}), the proper transform is a quartic surface $\mathcal{Q}$ with an equation of the form~(\ref{mainquartic}).
\end{proof}

\begin{remark}\label{rmk: nikulin 12}
By setting $\rho=0$ in Equation~(\ref{eqn:involution_quartic11}), we find a projective automorphism  $\mathbb{P}^3\dashrightarrow  \mathbb{P}^3$, given by
\beq
\label{eqn:involution_quartic}
\begin{split}
& \Big[ \mathbf{W} : \ \mathbf{X} : \ \mathbf{Y} :  \ \mathbf{Z} \Big] 
\ \mapsto \ \\
& \qquad \Big[ c_2\left(2 \mathbf{X}, \mathbf{W} \right) \mathbf{W}\mathbf{Z} : c_2\left(2 \mathbf{X}, \mathbf{W} \right)\mathbf{X}\mathbf{Z} : -c_2\left(2 \mathbf{X}, \mathbf{W}  \right) \mathbf{Y} \mathbf{Z} :  d_4\left(2 \mathbf{X}, \mathbf{W}  \right)  \Big],
\end{split}
\eeq
which induces a Nikulin involution on the minimal resolution of $\mathcal{Q}$ in Theorem~\ref{thm_12}.
\end{remark}
%
%
\subsection{Family of Picard rank 13}

By restricting the configuration of the lines $l_1, l_2$ and the uninodal quartic $\calN$, we can further specialize the family of $H\oplus D_6(-1) \oplus A_1(-1)^{\oplus 4}$-polarized K3 surfaces to a family of Picard rank $13$:
\begin{proposition}
\label{prop:rank13}
An $H\oplus D_6(-1) \oplus A_1(-1)^{\oplus 4}$-polarization extends to an $H\oplus D_8(-1) \oplus A_1(-1)^{\oplus 3}$-polarization if and only if for the corresponding double sextic $\mathcal{S}$ in Equation~(\ref{eqn:double_sextic12}) either $\ell_1$ is tangent to $\mathrm{n}$ or $\ell_1 \cap \ell_2 \cap \mathcal{N} \neq \emptyset$.
\end{proposition}
\begin{proof}
As in the proof of Proposition~\ref{prop:rank12}, we plug Equation~(\ref{eqn:params_rank12}) into Equation~(\ref{eqn:vgs}). The fiber of type $I_0^*$ extends to a fiber of type $I_2^*$ if and only if $b(u, v) = v^3 b'(u, v)$. For the double sextic in Equation~(\ref{eqn:double_sextic12}) we then have either $c_2(u, v)=v c_1(u, v)$ or $d_4(u, v)=v d_3(u, v)$ where $c_1$ and $d_3$ have degree 1 and 3, respectively.  In the first case, the line $\mathrm{V}(v)$ in the branch locus is tangent to the node of the quartic. In the second case, the two lines in the branch locus, i.e., $\mathrm{V}(v)$ and $\mathrm{V}(w)$, which intersect at $[1:0:0]$, now intersect on the quartic.
\end{proof}

Note that the pencils of lines described in Proposition~\ref{prop:fibration_rank12}~(i)-(iii) still induce Jacobian elliptic fibrations with the same types of singular fibers and the same Mordell-Weil group. 
The fibration in Proposition~\ref{prop:fibration_rank12}(iv) extends in the situation of Proposition~\ref{prop:rank13} as follows: when $\ell_1$ is tangent to $\mathrm{n}$, the fibration without section has singular fibers $III^* + I_0^* + 9 I_1$. 
For $\ell_1 \cap \ell_2 \cap \mathcal{N} \neq \emptyset$ the fibration has a section and is a standard fibration.

In terms of the generalized Inose quartic model (\ref{mainquartic}), the specialization corresponds to the situation when $(\xi, \omicron)=(0,1)$, or equivalently when $\chi^2=1$.
In this case, the lines $L_2^\pm$ become the lines
\beq 
L_2\colon \quad   \mathbf{W}= \mathbf{Z} =0 \,, \qquad \tilde{L}_2\colon \quad \mathbf{W}=2 \mathbf{X} + \gamma \varepsilon \mathbf{Z}=0\,. 
\eeq
Note that $L_7$ also coincides with the line $L_2$.
Again, we denote the corresponding pencils on the quartic surface $\mathcal{Q}$ by
$L_2(u, v)$ and $\tilde{L}_2(u, v)$.

\begin{remark}\label{rmk: sing 13}
    When $(\xi, \omicron)=(0,1)$, the singularities $P_1$ and $P_2$ in Remark~\ref{rmk: sing 12} are of types $A_5$ and $A_3$ respectively,
\end{remark}

We have the following:

\begin{theorem}
\label{thm_13}
Assume that the conditions of Remark~\ref{rmk: sing 12} are satisfied with $(\xi, \omicron)=(0,1)$, and let $L= H \oplus D_8(-1)  \oplus A_1(-1)^{\oplus 3}$. The minimal resolution of $\mathcal{Q}$ in Equation~(\ref{mainquartic}) is a K3 surface endowed with a canonical $L$-polarization. Conversely, an $L$-polarized K3 surface has a birational projective model~(\ref{mainquartic}) with $(\xi, \omicron)=(0,1)$. In particular,  Jacobian elliptic fibrations of the type determined in Theorem~\ref{thm:families} are attained as follows:

\beqn
\begin{array}{c|c|c|c|l}
\# 	&  \text{singular fibers} 		& \operatorname{MW} 		& \text{reducible fibers} 															& \text{pencil} \\
\hline
&&&&\\[-0.9em]
1	& I_2^* + 5 I_2 + 6 I_1		&  \mathbb{Z}/2\mathbb{Z}	& {\color{blue} \widetilde{D}_6} + {\color{magenta} \widetilde{A}_1}^{\oplus 5}  					& \text{residual surface intersection} \\ 
	&						&						&																			& \text{of $L_1(u, v)=0$ and $\mathcal{Q}$}\\
\hline
&&&&\\[-0.9em]
2	& I_0^* + I_2^* + I_2 + 8 I_1	&   \{ \mathbb{I} \} 			& {\color{blue} \widetilde{D}_4} + {\color{green} \widetilde{D}_6} + {\color{magenta} \widetilde{A}_1} 	& \text{residual surface intersection} \\ 
	&						&						&																			& \text{of $L_i(u, v)=0 \; (i=2,4,5,6)$ and $\mathcal{Q}$}\\
\hline
&&&&\\[-0.9em]
3 	& I_4^* + 3  I_2 + 8 I_1		&   \{ \mathbb{I} \} 			& {\color{blue} \widetilde{D}_8} + {\color{magenta} \widetilde{A}_1}^{\oplus 3} 					& \text{residual surface intersection} \\ 
	&						&						&																			& \text{of $\tilde{L}_2(u, v)=0$  and $\mathcal{Q}$}\\
\hline
&&&&\\[-0.9em]
4	& III^* + 4  I_2 + 7 I_1		&   \{ \mathbb{I} \} 			& {\color{blue} \widetilde{E}_7}+ {\color{magenta} \widetilde{A}_1}^{\oplus 4} 					& \text{intersection of quadric surfaces} \\ 
	&						&						&																			& \text{$\mathit{C}_i(u, v)=0  \, (i=4, 5, 6)$ and $\mathcal{Q}$}\\
\end{array}
\eeqn
\end{theorem}
\begin{proof}
The statements concerning fibrations~(1)-(3) can be proved in a similar way as in Theorem~\ref{thm_12} after setting $(\xi,\omicron)=(0,1)$. 

Pencils of quadratic surfaces, denoted by $\mathit{C}_n(u, v)=0$ with $[u:v] \in \mathbb{P}^1$ for $n= 4, 5, 6$, are constructed as follows. We set
\beq
\label{eqn:pencil_extra1}
\begin{split}
 \mathit{C}_4(u, v) & = \mu \big(\kappa u - \lambda v\big) \big(2 \mu \mathbf{X} - \nu \mathbf{W}\big) \big( 2 \kappa \mathbf{X} - \lambda \mathbf{W} + \gamma \varepsilon \kappa \mathbf{Z}\big) \\
 & - \kappa \big(\mu u - \nu v\big) \big(2 \kappa \mathbf{X} - \lambda \mathbf{W}\big) \big( 2 \mu \mathbf{X} - \nu \mathbf{W} + \gamma \varepsilon \mu \mathbf{Z}\big) \,,
\end{split} 
\eeq
which is invariant under the permutation of parameters $(\gamma, \delta)$ with $ (\varepsilon, \zeta)$.
The pencils $\mathit{C}_5$ and $\mathit{C}_6$ can be obtained by replacing the parameters 
$((\kappa, \lambda), (\mu, \nu))$ by $((\eta, \iota), (\mu, \nu))$ or $((\eta, \iota), (\kappa, \lambda))$.
By making the substitutions 
\beq
\label{eqn:substitution_bfd2}
\begin{array}{ll}
\mathbf{W} & = 2 \gamma^2 \varepsilon^2 \kappa^2 \mu^2 v^2  \big(\gamma u-\delta v\big)\big(\varepsilon u-\zeta v\big) x z  \,, \\ [0.4em]
\,\mathbf{X} & =  \gamma^2 \varepsilon^2 \kappa^2 \mu^2  v \big(\gamma u-\delta v\big)\big(\varepsilon u-\zeta v\big)  q_1(x, z, u, v) z\,, \\ [0.4em]
\,\mathbf{Y} & = \sqrt{2} \gamma \kappa \mu \big(\gamma u-\delta v\big)\big(\varepsilon u-\zeta v\big)  y z \,, \\  [0.4em]
\,\mathbf{Z} & = 2 q_2(x, z, u, v) \, q_3(x, z, u, v) \,, 
\end{array}
\eeq
in Equation~(\ref{mainquartic}) which is compatible with $C_4(u, v)=0$, and by using the polynomials
\beq
\begin{array}{ll}
q_1(x, z, u, v) & = u x - \gamma \varepsilon \kappa \mu v  \big(\gamma u - \delta v\big) \big(\varepsilon u - \zeta v\big)   \big(\kappa u - \lambda v\big)  \big(\mu u - \nu v\big)  z \,, \\ [0.4em]
q_2(x, z, u, v) & = x- \gamma \varepsilon \kappa^2 \mu v \big(\gamma u - \delta v\big) \big(\varepsilon u - \zeta v\big)  \big(\mu u - \nu v\big)  z \,, \\ [0.4em]
q_3(x, z, u, v) & = x- \gamma \varepsilon  \kappa \mu^2 v \big(\gamma u - \delta v\big) \big(\varepsilon u - \zeta v\big) \big(\kappa u - \lambda v\big)  z \,, 
\end{array}
\eeq
one obtains the Jacobian elliptic fibration~(4). Similarly, Jacobian elliptic fibrations with the same singular fibers are obtained from $C_5(u, v)$ and $C_6(u, v)$.
\end{proof}
%
%
\section{Dual graphs of rational curves}\label{sec_dual graph}
Following Kondo~\cite{MR1139659}, we define the \emph{dual graph} of smooth rational curves to be the multigraph whose set of vertices is the set of all smooth rational curves on a K3 surface such that the vertices $\Sigma, \Sigma'$ are joint by an $m$-fold edge if and only if their intersection product satisfies $\Sigma \circ \Sigma' = m$.  The following fact is well known; see \cites{MR2274533, Roulleau22}:
\begin{theorem}
\label{thm:rank10}
A very general $H \oplus N$-polarized K3 surface $\mathcal{X}$ satisfies the following:
\begin{enumerate}
\item the automorphism group of $\mathcal{X}$ is finite, i.e., $|\operatorname{Aut}(\mathcal{X})|<\infty$.  
\item $\mathcal{X}$ has exactly 18 smooth rational curves and the dual graph of rational curves is given by Figure~\ref{fig:pic10}, and the van Geemen-Sarti involution (\ref{eqn:VGS_involution}) acts as horizontal flip.
\end{enumerate}
\end{theorem}
\begin{figure}
	\centering
  	\scalebox{\MyScalePicBig}{%
    		\begin{tikzpicture}[rotate=90]
       		\input{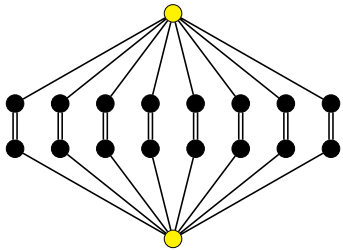}
    		\end{tikzpicture}}
	\caption{Rational curves on $\mathcal{X}$ with N\'eron-Severi lattice $H \oplus N$}
	\label{fig:pic10}
\end{figure}

We will now construct the dual graphs of all smooth rational curves for the K3 surfaces of Theorems~\ref{thm_11}, \ref{thm_12}, and ~\ref{thm_13}.  We will also provide the embeddings of the reducible fibers appearing in the tables of the theorems into the dual graphs. 
With each dual graphs, we will also visualise the Nikulin involution that is induced by the van Geemen-Sarti involution in Theorem~\ref{thm:rank10}.
In these figures, the components of the reducible fibers, given by extended Dynkin diagrams, are distinguished by the same colors that were used in the aforementioned theorems, and the classes of the section (and the 2-torsions section if applicable) are represented by yellow nodes.

By the adjunction formula, an irreducible curve on a K3 surface is a smooth rational curve if and only if its self-intersection number is $-2$.  As we will show, the Inose-type quartic normal form derived in Section~\ref{sec_fams} allows for the simple construction of rational curves as complete intersection curves. These curves turn out to give all rational curves on the corresponding K3 surfaces. The last statement will follow from a comparison with work by Roulleau in \cite{Roulleau22}: there, the number of all rational curves and their divisor classes were determined by a lattice theoretic method implementing an algorithm due to Vinberg.

\begin{remark}
Note that in the forthcoming figures in this section which feature the rational curves on a K3 surface in this section, some intersections of the curves are not shown. However, the intersection numbers left out can be easily computed using Equations~(\ref{eqn:divisor_classes_aux_11}), (\ref{eqn:divisor_classes_aux_12}) and (\ref{eqn:divisor_classes_aux_13}).
\end{remark}
%
%
 \subsection{Family of Picard rank 11}\label{sec_dual graph 11}
\par We determine the dual graph of smooth rational curves for the K3 surfaces $\mathcal{X}$ with N\'eron-Severi lattice $H \oplus D_4(-1) \oplus A_1(-1)^{\oplus 5}$  in Theorem~\ref{thm_11}. In this case, one has the following 90 rational curves: 
\begin{gather}
\label{eqn:curves_11}
  a_1, a_2, a_3, b_1, R_0, R_{13}, L_n,  R_n, L'_{n'}, R'_{n'},  \\
\label{eqn:curves_11_aux}  
   S_{n, n'}, \tilde{S}_{n,n'}, T_{mn, m'n'} \ \ \text{for $m, n, m', n' \in \{4, 5, 6, 7\}$ with $m<n, m'<n'$}.
\end{gather}
The lines $L_4, \dots, L'_7$ were already determined in Section~\ref{subsec: fibrarion 11}.   The two sets $\{ a_1, a_2, a_3 \}$ and $\{b_1 \}$ were mentioned in Remark~\ref{rmk:sing 11}. 

Lastly, the curves $R_\bullet, R'_\bullet, S_\bullet, \tilde{S}_\bullet, T_\bullet$ are obtained from singular complete intersection curves of higher (arithmetic) genus on the quartic surface $\mathcal{Q}$.  
When resolving the quartic surface~(\ref{mainquartic11}), these curves lift to smooth rational curves on the K3 surface, which by a slight  abuse of notation we shall denote by the same symbol.
In the following we give these singular complete intersections curves.

\beq
\begin{split}
R_0\colon \quad &\left\lbrace 
  \begin{array}{rcl}
   0 & = &   \mathbf{W} \, \\
   0 & = & 2 \rho \mathbf{Y}^2 \mathbf{Z}^2 + 4 \mathbf{X}^2 \big( c_2(1,0) \mathbf{Z}^2 + 2 \, e_3(1,0) \mathbf{X} \mathbf{Z} + 4 \, d_4(1,0) \mathbf{X}^2 \big),     \end{array} \right.\\
R_{13}\colon \quad &\left\lbrace 
  \begin{array}{rcl}
   0 & = &   \mathbf{W} \mathbf{Z}  c_2(2 \mathbf{X}, \mathbf{W})  + \rho \mathbf{Z}  e_3(2 \mathbf{X}, \mathbf{W}) + \rho  d_4(2 \mathbf{X}, \mathbf{W}) \, \\
   0 & = & 2 \rho \mathbf{Y}^2+  c_2(2 \mathbf{X}, \mathbf{W}),     \end{array} \right.\\
R_4\colon \quad &\left\lbrace 
  \begin{array}{rcl}
   0 & = &  2\eta \mathbf{X}-\iota \mathbf{W}\\
   0 & = &2 \eta^3 ( \mathbf{W} - \rho \mathbf{Z}) \mathbf{Y}^2 - \eta \, c_2(\iota, \eta) \mathbf{W}^2 \mathbf{Z} - e_3(\iota, \eta) \mathbf{W}^3 ,  \end{array} \right.\\
R'_4\colon \quad &\left\lbrace 
  \begin{array}{rcl}
   0 & = &  2\eta' \mathbf{X}-\iota' \mathbf{W}\\
   0 & = &2 (\eta')^4 \rho \mathbf{Z} \mathbf{Y}^2 + (\eta')^2 \, c_2(\iota', \eta') \mathbf{W}^2 \mathbf{Z} - \rho \, d_4(\iota', \eta') \mathbf{W}^3 .  \end{array} \right.
  \end{split}  
\eeq
Equations for $R_5$, $R_6$, and $R_7$ are obtained from $R_4$ by replacing $(\eta, \iota)$ by $(\kappa, \lambda)$, $(\mu, \nu)$, and $(\xi, \omicron)$, respectively. Similarly, the equations for $R'_5$, $R'_6$, and $R'_7$ are obtained from $R'_4$ by replacing  $(\eta', \iota')$ by $(\kappa', \lambda')$, $(\mu', \nu')$, and $(\xi', \omicron')$, respectively.

 We also have the following complete intersection curve:
\beq
\begin{split}
S_{4,4}\colon \quad &\left\lbrace 
 \begin{array}{rcl}
    0	& =	&   \iota \eta'  \big(\mathbf{W} - \rho \mathbf{Z} \big)  - \eta \big( 2 \eta' \mathbf{X} - \rho \iota' \mathbf{Z} \big)  \\[0.2em]
      0 	& = 	& 2 \rho\eta \eta'  \, \mathbf{Y}^2 +   \frac{(\eta')^2 ( 2 \eta \mathbf{X} - \iota \mathbf{W}) \, c_2(2 \mathbf{X}, \mathbf{W})}{ 2 \eta' \mathbf{X} - \iota' \mathbf{W}}\\ 
    	& & \phantom{2 \rho\eta \eta'  \, \mathbf{Y}^2} + 	\frac{\rho\eta'  \, (\eta \iota' - \iota \eta')   \, e_3(2 \mathbf{X}, \mathbf{W})}{ 2 \eta' \mathbf{X} - \iota' \mathbf{W}}  + \frac{ \rho^2 (\eta \iota' - \iota \eta')^2 d_4(2 \mathbf{X}, \mathbf{W})  }{(2 \eta \mathbf{X} - \iota \mathbf{W})(2 \eta' \mathbf{X} - \iota' \mathbf{W})},
\end{array} \right. 
\end{split}  
\eeq
where the second equation is a homogeneous polynomial of degree three. Equations for $S_{5,4}$, $S_{6,4}$, and $S_{7,4}$ are obtained from $S_{4,4}$ by replacing $(\eta, \iota)$ by $(\kappa, \lambda)$, $(\mu, \nu)$, and $(\xi, \omicron)$ respectively. 
The equations for $S_{n,5}$, $S_{n,6}$, and $S_{n,7}$ are obtained from $S_{n,4}$ for $n=4, 5, 6, 7$ by replacing  $(\eta', \iota')$ by $(\kappa', \lambda')$, $(\mu', \nu')$, and $(\xi', \omicron')$ respectively. This procedure generates the 16 rational curves $S_{4, 4}, S_{4, 5}, \dots, S_{7,6}, S_{7,7}$. 
Applying the Nikulin involution in Equation~(\ref{rmk: nikulin 11}), we obtain from $S_{4,4}$ the complete intersection curve:
\beqn
\begin{split}
\tilde{S}_{4,4}\colon \quad &\left\lbrace 
 \begin{array}{rcl}
    0	& =	&  \frac{1}{2 \eta' \mathbf{X} - \iota' \mathbf{W}} \Big( \eta' \mathbf{W} \mathbf{Z}  c_2(2 \mathbf{X}, \mathbf{W}) + \rho  \eta' \mathbf{Z}  e_3(2 \mathbf{X}, \mathbf{W})  \\
        & \phantom{+}  &  \phantom{\frac{1}{2 \eta' \mathbf{X} - \iota' \mathbf{W}} \Big( \eta' \mathbf{W} \mathbf{Z}  c_2(2 \mathbf{X}, \mathbf{W})} +\frac{ \rho (\eta'(2\eta\mathbf{X}-\rho \iota \mathbf{Z})-\eta\iota'(\mathbf{W}-\rho\mathbf{Z})) d_4(2 \mathbf{X}, \mathbf{W})  }{(2 \eta \mathbf{X} - \iota \mathbf{W})(2 \eta' \mathbf{X} - \iota' \mathbf{W})} \Big)\\[0.2em]
    0 	& = 	& 2 \rho\eta \eta'  \, \mathbf{Y}^2 +   \frac{(\eta')^2 ( 2 \eta \mathbf{X} - \iota \mathbf{W}) \, c_2(2 \mathbf{X}, \mathbf{W})}{ 2 \eta' \mathbf{X} - \iota' \mathbf{W}}  \\
    	&  \phantom{+} 	&\phantom{2 \rho\eta \eta'  \, \mathbf{Y}^2 } + \frac{\rho\eta'  \, (\eta \iota' - \iota \eta')   \, e_3(2 \mathbf{X}, \mathbf{W})}{ 2 \eta' \mathbf{X} - \iota' \mathbf{W}}  + \frac{ \rho^2 (\eta \iota' - \iota \eta')^2 d_4(2 \mathbf{X}, \mathbf{W})  }{(2 \eta \mathbf{X} - \iota \mathbf{W})(2 \eta' \mathbf{X} - \iota' \mathbf{W})} ,
\end{array} \right. 
\end{split}  
\eeqn
where the first and second equation are again homogeneous polynomials. Equations for $\tilde{S}_{5,4}, \dots, \tilde{S}_{7,7}$  are then obtained from $\tilde{S}_{4,4}$ in the same way as before.
\par For  $m, n, m', n' \in \{4, 5, 6, 7\}$ with $m<n, m'<n'$, we also introduce the following complete intersection curves:
\beqn
\begin{split}
T_{mn,m'n'}\colon \quad &\left\lbrace 
 \begin{array}{rcl}
    0	& =	&  \frac{ Q_{m'n'}(2  \mathbf{X}, \mathbf{W})  + \rho P_{mn}(2  \mathbf{X}, \mathbf{W}) (\mathbf{W} -\rho \mathbf{Z})}{\mathbf{W}} \\[0.2em]
    0 	& = 	&  2 \rho Y^2  + \frac{  \rho^2 P_{mn}(2  \mathbf{X}, \mathbf{W}) c_2(2 \mathbf{X}, \mathbf{W}) }{Q_{m'n'}(2  \mathbf{X}, \mathbf{W}) } 
    +  \frac{\rho (\rho^2 P_{mn}(2  \mathbf{X}, \mathbf{W}) -Q_{m'n'}(2  \mathbf{X}, \mathbf{W})) e_3(2 \mathbf{X}, \mathbf{W}) }{Q_{m'n'}(2  \mathbf{X}, \mathbf{W})  \mathbf{W}}  \\
   	 & &\phantom{2 \rho Y^2} + \frac{(\rho^2 P_{mn}(2  \mathbf{X}, \mathbf{W}) -Q_{m'n'}(2  \mathbf{X}, \mathbf{W}))^2 d_4(2 \mathbf{X}, \mathbf{W})}{P_{mn}(2  \mathbf{X}, \mathbf{W}) Q_{m'n'}(2  \mathbf{X}, \mathbf{W})  \mathbf{W}^2}  ,
\end{array} \right. 
\end{split}  
\eeqn
where the first and second equation are homogeneous polynomials. For $m=4, n=5$ and $m'=4, n'=5$, we set
\beq
\begin{split}
  P_{45}(2  \mathbf{X}, \mathbf{W})  &=  \mu \xi (2 \eta \mathbf{X} - \iota \mathbf{W}) (2 \kappa \mathbf{X} - \lambda \mathbf{W}), \\
  Q_{45}(2  \mathbf{X}, \mathbf{W})  &= \mu' \xi' (2 \eta' \mathbf{X} - \iota' \mathbf{W}) (2 \kappa' \mathbf{X} - \lambda' \mathbf{W}),
\end{split}  
\eeq
and equations for $T_{46,45}, \dots, T_{67,45}, T_{67,46}, \dots, T_{67,67}$ are obtained from $T_{45,45}$ by interchanging parameters in an analogous manner as above.  This procedure generates 36 rational curves. Applying the Nikulin involution in Equation~(\ref{rmk: nikulin 11}), we obtain from the curve $T_{m n, m' n'}$ the curve $T_{k l, k'l'}$ with $\{k, l, m, n\} = \{k', l', m', n'\}= \{4, 5, 6, 7\}$ and $m<n$, $k<l$, $m'<n'$, $k'<l'$.

  We have the following:
\begin{theorem}
\label{thm:polarization11}
\leavevmode
\begin{enumerate}
\item The rational curves $a_1, a_2, a_3, L_4, L_5, L_6, L_7, R_0, R'_4, R'_5, R'_6$ generate the N\'eron-Severi lattice $H \oplus D_4(-1) \oplus A_1(-1)^{\oplus 5}$.  

\item For the K3 surfaces in Theorem~\ref{thm_11} the dual graph of all smooth rational curves is shown in Figure~\ref{fig:pic11_2folds}. 

\item On Figure~\ref{fig:pic11_2folds} the Nikulin involution acts as a horizontal flip, exchanging the nodes $a_1$ and $b_3$.

\item As divisor classes in the N\'eron-Severi lattice, one has the following relations:
\beq
\label{eqn:divisor_classes_aux_11}
\begin{gathered}
 b_1			\sim 2 a_1 + L_4 + L_5 + L_6 + L_7 - R_0, \quad
 R'_7 		\sim a_2  - L_7 + R_0, \\
 S_{n,n'}	 	\sim 2a_1 + L_k + L_l + L_m - R_0 + R'_{n'},\\
 T_{mn, m'n'}	\sim 2a_1 + L_k +L_l - R_0 + R'_{m'} + R'_{n'},\\
 R_{13}		\sim 2 \mathpzc{D}_{11} - b_1,\quad
 \tilde{S}_{n,n'}	\sim 2 \mathpzc{D}_{11} - S_{n,n'}, \quad
  T_{mn, m'n'}	\sim 2 \mathpzc{D}_{11} - T_{kl, k'l'},
\end{gathered}
\eeq
for $\{ k, l, m, n \} = \{ k', l', m', n' \} =  \{4, 5, 6, 7\}$ and $m<n$, $k<l$, $m'<n'$, $k'<l'$, and the divisor
\beqn
 \mathpzc{D}_{11} =  a_1 + 2 a_2 + a_3 + R_0,
\eeqn
which is  invariant under the action of the Nikulin involution.

\item Figures~\ref{fig:pic11_1}-\ref{fig:pic11_2} determine the embeddings of the  reducible fibers from Theorem~\ref{thm_11}.
 \end{enumerate}
\end{theorem}
\begin{figure}
\scalebox{0.35}{%
    	\begin{tikzpicture}
       		\input{ pic_rank_11_0.tex}
    	\end{tikzpicture}}
	\caption{Dual graph of all smooth rational curves with 1-fold (thin) and some 2-fold (thick) edges for $\mathrm{NS}(\mathcal{X}) = H \oplus D_4(-1) \oplus A_1(-1)^{\oplus 5}$}
 	\label{fig:pic11_2folds}
\end{figure}

\begin{proof}
Items (1), (3) and (4) can be shown by explicit computation.
For (2),
the comparison of the divisor classes for the constructed rational curves with \cite{Roulleau22}*{Sec.~11.4} shows that they constitute all rational curves. 
Finally for (5), there is only one way of embedding the reducible fibers of fibration~(1) in Theorem~\ref{thm_11} into the graph in Figure~\ref{fig:pic11_2folds}; see  Figure~\ref{fig:pic11_1}.   
For all other Jacobian fibrations, there is a second embedding of the reducible fibers, related by the action of the Nikulin involution. The embeddings are depicted in Figure~\ref{fig:pic11_2}.  
\end{proof}
%
%
 \subsection{Family of Picard rank 12}\label{subsec:curves 12}

We also determine the dual graph of all rational curves for the K3 surfaces in Theorem~\ref{thm_12} with N\'eron-Severi lattice $H \oplus D_6(-1) \oplus A_1(-1)^{\oplus 4}$. In this situation one has the following 59 rational curves: 
\begin{gather}
\label{eqn:curves_12}
  a_1, a_2, a_3, b_1, b_2, b_3, L_1, L^\pm_2, R_1, R_3, L_n,  R_n,  \\
 \label{eqn:curves_12_aux}
  S^\pm_n, \tilde{S}^\pm_n,  T^\pm_{mn}, \tilde{T}^\pm_{mn} \ \ \text{for $m, n \in \{4, 5, 6, 7\}$ with $m<n$}.
\end{gather}
Again, the lines $L_1, L_2^{\pm}, L_n$ for $n = 4, 5, 6, 7$ were mentioned in Section~\ref{subsec: fam 12}.
The two sets of lines $\{a_1, a_2, a_3\}$ and $\{b_1, b_2, b_3\}$ comes from resolving the singularities at $\mathrm{P}_1$ and $\mathrm{P}_2$ in Remark~\ref{rmk: sing 12}, respectively. The remaining curves can be given as complete intersections.
We set
\beqn
\begin{split}
R_1\colon \quad  &\left\lbrace 
  \begin{array}{rcl}
   0 & = &  2\gamma \mathbf{X}-\delta \mathbf{W} \\
   0 & = &  2 \gamma^4\mathbf{Y}^2 \mathbf{Z}  - \gamma  e_3(\delta, \gamma) \mathbf{W}^2 \mathbf{Z}  - d_4(\delta, \gamma) \mathbf{W}^3 ,  \end{array} \right. \\
R_4\colon \quad &\left\lbrace 
  \begin{array}{rcl}
   0 & = &  2\eta \mathbf{X}-\iota \mathbf{W}\\
   0 & = & 2 \eta^3 \mathbf{Y}^2 - \eta \, c_2(\iota, \eta)  \mathbf{W} \mathbf{Z} -  e_3(\iota, \eta)   \mathbf{W}^2  .  \end{array} \right.
  \end{split}  
\eeqn
where the polynomials $c_2, e_3$ and $d_4$ are determined in Equation~\ref{eqn:polys_2sextics}. 
An equation for $R_3$ is obtained from $R_1$ by replacing $(\gamma, \delta)$   by $(\varepsilon, \zeta)$, and $R_5$, $R_6$, and $R_7$ are obtained from $R_4$ by replacing $(\eta, \iota)$ by $(\kappa, \lambda)$, $(\mu, \nu)$ and $(\xi, \omicron)$ respectively.  That is, they are related by the symmetries in Remark~\ref{rmk:symmetries 12}.
We also have
\beqn
\begin{split}
S^\pm_4\colon \quad &\left\lbrace 
 \begin{array}{rcl}
    0	& =	&  -\iota (1 \pm \chi) \mathbf{W} + 2 \eta (1 \pm \chi) \mathbf{X} + 2 \gamma \varepsilon \eta \mathbf{Z}  \\[0.2em]
    0 	& = 	& 4 \gamma  \varepsilon  \eta  ( 1 \pm \chi ) \mathbf{Y}^2 +  \frac{1}{\mathbf{W}} \Big( ( 1 \pm \chi )^2  \big( 2 \eta \mathbf{X} - \iota \mathbf{W} \big) \, c_2(2 \mathbf{X}, \mathbf{W}) \\
    	&	&\phantom{4 \gamma  \varepsilon  \eta  ( 1 \pm \chi ) \mathbf{Y}^2 +\frac{1}{\mathbf{W}} \Big(} - 2  \gamma  \varepsilon  \eta  ( 1 \pm \chi )  \, e_3(2 \mathbf{X}, \mathbf{W}) + 4  \gamma^2  \varepsilon^2  \eta^2 \frac{ d_4(2 \mathbf{X}, \mathbf{W})  } { 2 \eta \mathbf{X} - \iota \mathbf{W} } \Big) ,
\end{array} \right.\\
\tilde{S}^\pm_4\colon \quad &\left\lbrace 
\begin{array}{rcl}
    0 	& = 	&   (1 \pm \chi) \, c_2(2 \mathbf{X}, \mathbf{W})  \, \mathbf{Z} + \frac{ 2 \gamma \varepsilon \eta \, d_4(2 \mathbf{X}, \mathbf{W})  } { 2 \eta \mathbf{X} - \iota \mathbf{W} }  \\
    0 	& = 	& 4 \gamma  \varepsilon  \eta  ( 1 \pm \chi ) \mathbf{Y}^2 +  \frac{1}{\mathbf{W}} \Big( ( 1 \pm \chi )^2  \big( 2 \eta \mathbf{X} - \iota \mathbf{W} \big) \, c_2(2 \mathbf{X}, \mathbf{W}) \\
    	& &\phantom{4 \gamma  \varepsilon  \eta  ( 1 \pm \chi ) \mathbf{Y}^2 + \frac{1}{\mathbf{W}} \Big(}	 -2  \gamma  \varepsilon  \eta  ( 1 \pm \chi )  \, e_3(2 \mathbf{X}, \mathbf{W}) + 4  \gamma^2  \varepsilon^2  \eta^2 \frac{ d_4(2 \mathbf{X}, \mathbf{W})  } { 2 \eta \mathbf{X} - \iota \mathbf{W} } \Big) .
\end{array} \right.
\end{split}  
\eeqn%
where the second equation for $S_4^{\pm}$ is a homogeneous polynomial of degree two, and
$\chi$ is as in Equation~\ref{eqn:chi} with $\chi^2 \neq 0,1$.
Equations for $S^\pm_5$, $S^\pm_6$, and $S^\pm_7$ are obtained from $S^\pm_4$ by replacing $(\eta, \iota) \leftrightarrow (\kappa, \lambda)$, $(\eta, \iota)  \leftrightarrow (\mu, \nu)$, and $(\eta, \iota)  \leftrightarrow (\xi, \omicron)$, respectively; in the same way, $\tilde{S}^\pm_5$, $\tilde{S}^\pm_6$, and $\tilde{S}^\pm_7$ are obtained from $\tilde{S}^\pm_4$. 
Divisors $S^+_n$ and $S^-_n$, as well as $\tilde{S}^+_n$ and $\tilde{S}^-_n$ for $n=4, 5, 6, 7$ are interchanged by the sign change $\chi \mapsto - \chi$.
Note that $\tilde{S}^\pm_4$ is obtained from $S_4^\pm$ by applying the Nikulin involution in Remark~\ref{rmk: nikulin 12}.
Divisors $S^\pm_n$ and $\tilde{S}^\pm_n$ as well as $L_n$ and $R_n$ for $n=4, 5, 6, 7$  are also related by the action of the Nikulin involution.
Finally, we set
\beqn
\begin{split}
T^\pm_{45}\colon \quad &\left\lbrace 
 \begin{array}{rcl}
    0	& =	&  (1 \mp \chi) \big(2 \gamma \mathbf{X} - \delta \mathbf{W}\big) \mathbf{Z} + 2 \gamma  \mu  \xi (1 \mp \chi)  \big(2 \eta \mathbf{X} - \iota \mathbf{W}\big)  \big(2 \kappa\mathbf{X} - \lambda \mathbf{W}\big)    \\[0.2em]
    0 	& = 	& 4 \gamma  \varepsilon  \kappa \mu \xi  ( 1 \mp \chi ) \mathbf{Y}^2 +  \frac{1}{\mathbf{W}} \left( \frac{4 \gamma^2 \varepsilon \kappa \mu^2 \xi ^2 ( 2 \kappa \mathbf{X} - \lambda \mathbf{W}) (2 \eta \mathbf{X} - \iota \mathbf{W})  \, c_2(2 \mathbf{X}, \mathbf{W})}{ 2 \gamma \mathbf{X} - \delta \mathbf{W} } \right. \\
    	& &\phantom{4 \gamma  \varepsilon  \kappa \mu \xi  ( 1 \mp \chi ) \mathbf{Y}^2} -\left.  2  \gamma  \varepsilon  \kappa \mu \xi  ( 1 \mp \chi )  \, e_3(2 \mathbf{X}, \mathbf{W}) +  \frac{\varepsilon \kappa (1 \mp \chi)^2  (2 \gamma \mathbf{X} - \delta \mathbf{W}) \, d_4(2 \mathbf{X}, \mathbf{W})}  {(2 \eta \mathbf{X} - \iota \mathbf{W}) (2 \kappa \mathbf{X} - \lambda \mathbf{W}) } \right) .
\end{array} \right. 
\end{split}  
\eeqn
Equations for $T^\pm_{46}$, $T^\pm_{47}$, $\dots$,  $T^\pm_{67}$ are obtained from $T^\pm_{45}$ by interchanging parameters in an analogous manner as above.  
Applying the Nikulin involution in Remark~\ref{rmk: nikulin 12}, we also obtain  $\tilde{T}^\pm_{45}, \dots, \tilde{T}^\pm_{67}$. 
For the latter, applying the Nikulin involution essentially interchanges $(\gamma, \delta)$ and $ (\varepsilon, \zeta)$. 

We have the following:
\begin{theorem}\leavevmode
\label{thm:polarization12}
\begin{enumerate}
\item The rational curves $a_1, a_2, a_3, b_2, L_1, L_2^\pm, L_4, L_5, L_6, L_7, R_1$ generate the N\'eron-Severi lattice  $H \oplus D_6(-1) \oplus A_1(-1)^{\oplus 4}$.  
\item For the K3 surfaces in Theorem~\ref{thm_12} the dual graph of all smooth rational curves is shown in  Figure~\ref{fig:pic12_2folds}.
\item On Figure~\ref{fig:pic12_2folds} the Nikulin involution acts as a horizontal flip, exchanging the nodes $a_1$ and $b_2$.
\item As divisor classes in the N\'eron-Severi lattice, one has the following relations:
\beq
\label{eqn:divisor_classes_aux_12}
\begin{gathered}
 S^{\pm}_n 		\sim L_k + L_l + L_m - L^\pm_2 + 2 a_1 + a_2, \quad \tilde{S}^\pm_n 	\sim 2 \mathpzc{D}_{12} - S^{\pm}_n,\\
 T^{\pm}_{mn} 		\sim L_m + L_n  + R_1  -  L_2^{\pm} + 2 a_1 + a_2 , \quad
 \tilde{T}^\pm_{mn} 	\sim 2 \mathpzc{D}_{12} - T^{\pm}_{mn},
\end{gathered}
\eeq
for $\{k, l, m, n\} = \{4, 5, 6, 7\}$, $m < n$, and the divisor
\beqn
 \mathpzc{D}_{12} =  a_1 + 2 a_2 + 3 a_3 + b_2 + 2 L_1 + L^+_2 + L^-_2 ,
\eeqn
which is  invariant under the action of the Nikulin involution.

\item Figures~\ref{fig:pic12_1}-\ref{fig:pic12_3} determine the embeddings of the  reducible fibers from Theorem~\ref{thm_12}.
\end{enumerate}
\end{theorem}

\begin{figure}
\scalebox{0.5}{%
    	\begin{tikzpicture}
       		\input{ pic_rank_12_0.tex}
    	\end{tikzpicture}}
	\caption{Dual graph of all smooth rational curves with 1-fold (thin) and a few 2-fold (thick) edges for $\mathrm{NS}(\mathcal{X}) = H \oplus D_6(-1) \oplus A_1(-1)^{\oplus 4}$}
 	\label{fig:pic12_2folds}
\end{figure}

\begin{proof}
Items (1), (3), (4) can be shown by direct computation.
The comparison of the divisor classes for the constructed rational curves with \cite{Roulleau22}*{Sec.~12.3} gives (2).
For (5), there is only one way of embedding the reducible fibers of fibration~(1) in Theorem~\ref{thm_12} into the graph in Figure~\ref{fig:pic12_2folds}. 
It is depicted in Figure~\ref{fig:pic12_1}.
For the remaining fibrations, there are always two distinct embeddings of the reducible fibers for each Jacobian elliptic fibration, related by the action of the Nikulin involution.  
They are depicted in Figures~\ref{fig:pic12_2} and \ref{fig:pic12_3}. 
\end{proof}
%
%
\subsection{Family of Picard rank 13}
Finally, we determine the dual graph of smooth rational curves for the K3 surfaces in Theorem~\ref{thm_13} with N\'eron-Severi lattice $H \oplus D_8(-1)  \oplus A_1(-1)^{\oplus 3}$. In this situation one has the following 39 rational curves: 
\begin{gather}
\label{eqn:curves_13}
  a_1, \dots, a_5, b_1, b_2, b_3, L_1, L_2, \tilde{L}_2, L_n, R_1, R_3, R_n,  \\
  \label{eqn:curves_13_aux}  
  S_2, S_n,  \tilde{S}_2, \tilde{S}_n,  T_{i, n}, \tilde{T}_{i, n} \ \text{for $i \in \{1, 3\}$, $n \in \{4, 5, 6\}$}.
\end{gather}
The lines  $L_1, L_2, \tilde{L}_2, L_4, L_5, L_6$ were already determined in Section~\ref{subsec: fam 12}.  
The two sets $\{ a_1, a_2, a_3, a_4, a_5 \}$ and $\{b_1, b_2, b_3\}$ denote the curves appearing when resolving the singularities at $\mathrm{P}_1$ and $\mathrm{P}_2$ mentioned in Remark~\ref{rmk: sing 13}. 

The rest of the curves can again be given as complete intersections.
From the curves $R_1$ and $R_3$ in Section~\ref{subsec:curves 12}, the condition $(\xi, \omicron)=(0,1)$ for the equation of $\calQ$ restricts them to curves of the same degree which we will denote by the same symbols.
Setting  $\chi= 1$ in the equations for $S^+_n$ with $n= 4, 5, 6$ in Picard number 12 restricts the divisors that will are denoted by $S_n$. 
The equations for $S^-_n$ become reducible combinations of $L_4, L_5, L_6$.  The same divisors are obtained if one sets $\chi=-1$ in the equations for $S^-_n$. 
Similarly, we obtain the restrictions  $\tilde{S}_n$ of $\tilde{S}^\pm_n$.  
Note that $\tilde{S}_n$ is obtained from $S_n$ by applying the Nikulin involution in Remark~\ref{rmk: nikulin 12}, and the equations for $S_n$ (resp. $\ti{S}_n$)  for $n = 4, 5,6$ are related by interchanging parameters.
Moreover, we have
\beqn
\begin{split}
S_2\colon \quad &\left\lbrace 
  \begin{array}{rcl}
    0 & = &  \eta \kappa \mu \mathbf{W} - \mathbf{Z}  \\
    0 & = &  \eta^2 \kappa^2 \mu^2 c_2(2 \mathbf{X}, \mathbf{W}) +  \frac{\eta \kappa \mu  \, e_3(2 \mathbf{X}, \mathbf{W})  \mathbf{W} +  d_4(2 \mathbf{X}, \mathbf{W})  }{\mathbf{W}^2}  -2 \eta \kappa \mu  \mathbf{Y}^2  ,  \end{array} \right.\\
\tilde{S}_2\colon \quad &\left\lbrace 
  \begin{array}{rcl}
   0 & = & \eta \kappa \mu  \, c_2(2 \mathbf{X}, \mathbf{W}) \mathbf{Z} - \frac{d_4(2 \mathbf{X}, \mathbf{W})  }{\mathbf{W}}\\
   0 & = &  \eta^2 \kappa^2 \mu^2 c_2(2 \mathbf{X}, \mathbf{W}) +  \frac{\eta \kappa \mu  \, e_3(2 \mathbf{X}, \mathbf{W})}{  \mathbf{W}} + \frac{  d_4(2 \mathbf{X}, \mathbf{W})  }{\mathbf{W}^2 (2\eta \mathbf{X} - \iota \mathbf{W})}  -2 \eta \kappa \mu  \mathbf{Y}^2 ,  \end{array} \right. \\
   T_{1,4}\colon \quad &\left\lbrace 
  \begin{array}{rcl}
    0 & = & \gamma \kappa \mu \big(2 \epsilon \mathbf{X} - \zeta \mathbf{W} \big)\mathbf{Z}-  \big( 2 \kappa \mathbf{X} - \lambda \mathbf{W}  \big) \big( 2 \mu \mathbf{X} - \nu \mathbf{W}  \big)\\
    0 & = &- \frac{ ( 2 \kappa \mathbf{X} - \lambda \mathbf{W} )  ( 2 \mu \mathbf{X} - \nu \mathbf{W} ) \, c_2(2 \mathbf{X}, \mathbf{W})}{\mathbf{W} (2 \varepsilon \mathbf{X} - \zeta \mathbf{W})} +   \frac{\gamma  \kappa \mu   \, e_3(2 \mathbf{X}, \mathbf{W})}{\mathbf{W}}\\
    & &  \phantom{- \frac{ ( 2 \kappa \mathbf{X} - \lambda \mathbf{W} )  ( 2 \mu \mathbf{X} - \nu \mathbf{W} ) \, c_2(2 \mathbf{X}, \mathbf{W})}{\mathbf{W} (2 \varepsilon \mathbf{X} - \zeta \mathbf{W})}} -\frac{\gamma^2 \kappa^2 \mu^2 (2 \varepsilon \mathbf{X} - \zeta \mathbf{W}) \, d_4(2 \mathbf{X}, \mathbf{W})  } { (2 \kappa \mathbf{X} - \lambda \mathbf{W})   (2 \mu \mathbf{X} - \nu \mathbf{W}) }   - 2 \gamma \kappa \mu \mathbf{Y}^2. \end{array} \right. \\
   \tilde{T}_{1,4}\colon \quad &\left\lbrace 
  \begin{array}{rcl}
    0 & = &  \big(2 \gamma \mathbf{X} - \delta \mathbf{W} \big)\mathbf{Z} -  \gamma \kappa \mu \big( 2 \eta \mathbf{X} - \iota \mathbf{W}  \big) \mathbf{W}  \\
    0 & = &  \frac{\gamma^2 \kappa^2 \mu^2  ( 2 \eta \mathbf{X} - \iota \mathbf{W} ) \, c_2(2 \mathbf{X}, \mathbf{W})}{2 \gamma \mathbf{X} - \delta \mathbf{W}} +   \frac{\gamma  \kappa \mu   \, e_3(2 \mathbf{X}, \mathbf{W})}{\mathbf{W}} +   \frac{ (2 \gamma \mathbf{X} - \delta \mathbf{W}) \, d_4(2 \mathbf{X}, \mathbf{W})  } { \mathbf{W}^2 (2 \eta \mathbf{X} - \iota \mathbf{W}) }   - 2 \gamma \kappa \mu \mathbf{Y}^2 ,  \end{array} \right. 
\end{split}  
\eeqn%
where all equations are polynomial by construction.  
The curves $S_n$ and $\tilde{S}_n$ (resp. $T_{1,4}$ and $\tilde{T}_{1,4}$) are related by the action of the Nikulin involution. 
Equations for $T_{3,4}$ and $\tilde{T}_{3,4}$ are obtained from $T_{1,4}$ and $\tilde{T}_{1,4}$ by replacing the parameters $(\gamma, \delta) $ by $(\varepsilon, \zeta)$.
Moreover, equations for $T_{m,5}$ and $\tilde{T}_{m,6}$ are obtained from $T_{m,4}$ and $\tilde{T}_{m,4}$ for $m=1, 3$ by replacing $(\eta, \iota)$ by $(\kappa, \lambda)$ and $(\mu, \nu)$ respectively. 
These divisors appeared above are in connection with the pencils $C_n(u, v)$ for $n=4, 5, 6$ in the proof of Theorem~\ref{thm_13}.

\begin{theorem}
\label{thm:polarization13}
\leavevmode
\begin{enumerate}
\item The rational curves $a_1, \dots, a_5, b_1, b_2, L_2, L_4, L_5, L_6, R_1, R_3$ generate the N\'eron-Severi lattice $H \oplus D_8(-1)  \oplus A_1(-1)^{\oplus 3}$. 
\item For the K3 surfaces in Theorem~\ref{thm_13} the dual graph of all smooth rational curves is shown in  Figure~\ref{fig:pic13_2folds}.
\item On Figure~\ref{fig:pic13_2folds} the Nikulin involution acts as a horizontal flip, exchanging the nodes $a_1$ and $b_2$.
\item As divisor classes in the N\'eron-Severi lattice, one has the following relations:
\beq
\label{eqn:divisor_classes_aux_13}
\begin{gathered}
 \tilde{T}_{i, l}	\sim 2a_1 + a_2 -L_2 + L_m + L_n + R_j, \quad T_{i, l}	 	\sim \mathpzc{D}_{13} - \tilde{T}_{i, l}\\
 \tilde{S}_n	\sim 2a_1 + a_2 -L_2 + L_n + R_1 + R_3, \quad S_n			\sim \mathpzc{D}_{13} - \tilde{S}_n\\
 S_2			\sim 2a_1 + a_2 - L_2 + L_4 + L_5 + L_6, \quad \tilde{S}_2 	\sim \mathpzc{D}_{13} - S_2,
\end{gathered}
\eeq
for $\{ i, j \} = \{1, 3\}$, $\{l, m, n \} = \{ 4, 5, 6\}$, and the divisor
\beqn
 \mathpzc{D}_{13} =  3 a_1 + 3 a_2 + 3 a_3 + 2 a_4 + a_5 - 2 b_1 - b_2 + L_2 + L_4 +L_5 + L_6 - R_1 + R_3 ,
\eeqn 
which is invariant under the action of the Nikulin involution
\item Figures~\ref{fig:pic13_1}-\ref{fig:pic13_4} determine the embeddings of the  reducible fibers from Theorem~\ref{thm_13}.
\end{enumerate}
\end{theorem}

\begin{figure}
\scalebox{0.52}{%
    	\begin{tikzpicture}
       		\input{ pic_rank_13_0.tex}
    	\end{tikzpicture}}
	\caption{Dual graph of all smooth rational curves with 1-fold (thin) and some 2-fold (thick) edges for $\mathrm{NS}(\mathcal{X}) = H \oplus D_8(-1) \oplus A_1(-1)^{\oplus 3}$}
 	\label{fig:pic13_2folds}
\end{figure}
\begin{proof}
Items (1), (3), (4) can be shown by direct computation.
The comparison of the divisor classes for the constructed rational curves with \cite{Roulleau22}*{Sec.~13.2} gives (2). 
For (5), 
there is only one way of embedding the reducible fibers of fibration~(1), which is depicted in Figure~\ref{fig:pic13_1}. 
For the remaining fibrations there are  two embeddings in each case, related by the action of the Nikulin involution.  They are depicted in Figures~\ref{fig:pic13_2} to \ref{fig:pic13_4}. 
\end{proof}

Before we end the section, We have the following:
\begin{theorem}
\label{thm:polarizations}
The polarization divisor $\mathpzc{H}$ of $\mathcal{Q}$ is as follows:
\beq
\label{linepolariz}
\begin{array}{rcll}
\mathpzc{H} &= & L_4 + L_5 + L_6 + L_7 +   3 a_1 +2 a_2 + a_3   &\text{in Theorem~\ref{thm_11}, } \\
\mathpzc{H} &= & L_4 + L_5 + L_6 + L_7 +   3 a_1 +2 a_2 + a_3   &\text{in Theorem~\ref{thm_12}, } \\
\mathpzc{H} & = & L_2 + L_4 + L_5 + L_6 +   3 a_1 +3 a_2 + 3 a_3 + 2 a_4 + a_5 &\text{in Theorem~\ref{thm_13}.}
\end{array}
\eeq
One has $\mathpzc{H}^2=4$ and $\mathpzc{H} \circ F=3$ where $F$ is the smooth fiber class of any elliptic fibration that is obtained as the intersection with a pencil of lines.
\end{theorem}
\begin{proof}
We will prove the statement for Picard number 12. The other cases are analogous. By computing all intersection numbers of the rational curves, we may express their divisor classes as linear combinations with integer coefficients in a basis of the lattice.
Let us look at fibration~(2)  in Theorem~\ref{thm_12}, i.e., the standard fibration and compute a basis of the polarization lattice. 
We observe that the nodes $L_1$ and $a_2$ are the extra nodes of the two extended Dynkin diagrams of $\widetilde{D}_4$.  It follows that $L^+_2, L_4, \dots,  L_7$, $a_1, a_3$, $b_1, b_2, b_3$, and the fiber class $F$ of the standard fibration form a basis in $\mathrm{NS}(\mathcal{X})$. For example, taking a fiber class in Figure~\ref{fig:pic12_2_1} we have
\beq
 F = 2 a_1 + a_2 + L_5 + L_6 + L_7.
\eeq
Thus, the polarizing divisor can be written as a linear combination
\beq
\label{eqn:lin_comb12}
\mathpzc{H} = f \,F  + l_2 L^+_2 + \sum_{i=4}^7 l_i \, L_i +  \sum_{i=1}^3 \beta_i \, b_i  
+ \alpha_1 \, a_1 + \alpha_3 \, a_3  \,.
\eeq
We use $\mathpzc{H} \circ a_i = \mathpzc{H} \circ b_i = 0$ for $i=1, 2, 3$, $\mathpzc{H} \circ L_1=\mathpzc{H} \circ L^\pm_2=1$, and  $\mathpzc{H} \circ L_j=1$ for $j=4, \dots, 7$. We obtain a linear system of equations for the coefficients in Equation~(\ref{eqn:lin_comb12}) whose unique solution is given by Equation~(\ref{linepolariz}). 
\end{proof}%
%
%

\newpage
\begin{figure}
	\centering
  	\scalebox{0.1}{%
    		\begin{tikzpicture}
       			\input{ pic_rank_11_1.tex}
    		\end{tikzpicture}}
	\caption{Fibration with reducible fibers ${\color{magenta} \widetilde{A}_1}^{\oplus 9}$ and {\color{yellow} sections}}
  	\label{fig:pic11_1}
\end{figure}%
\begin{figure}
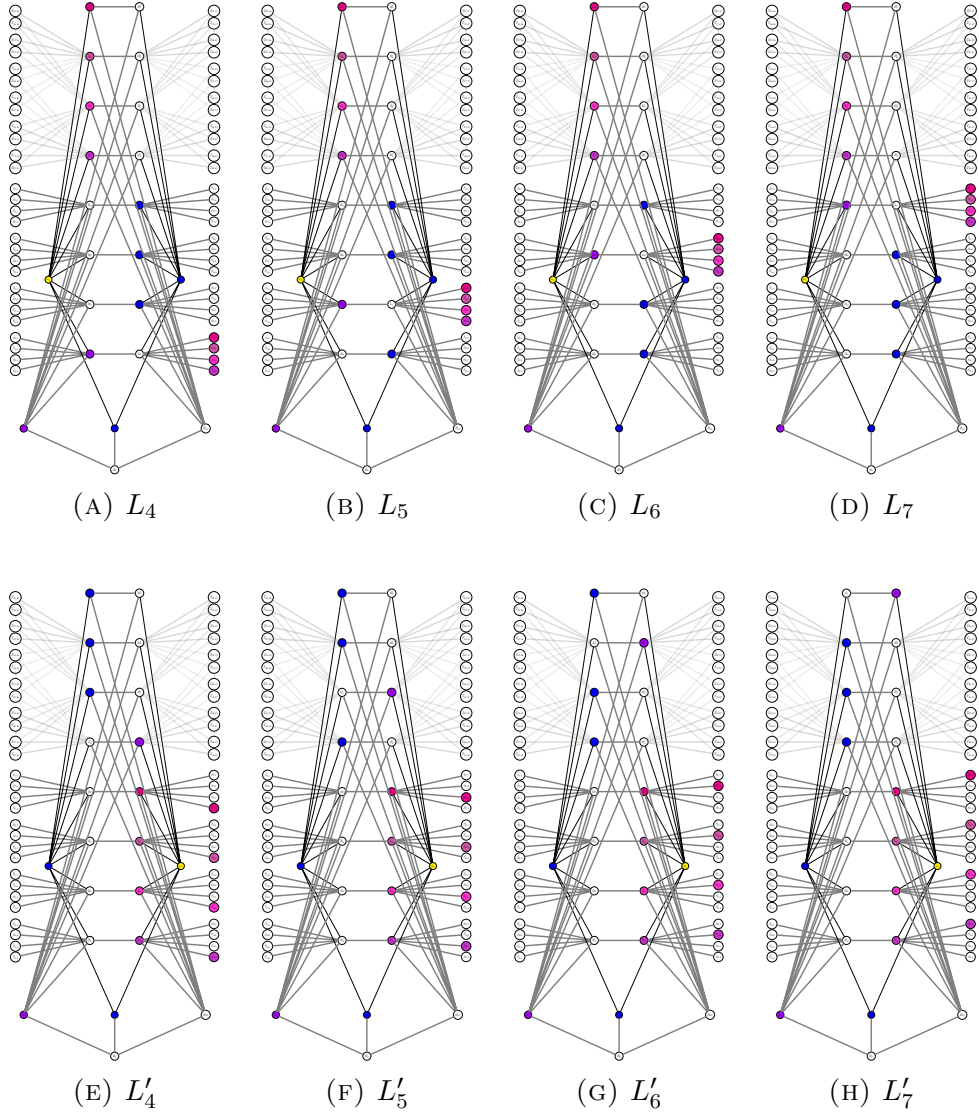

\begin{subfigure}{.22\textwidth}  
  	\centering
	\scalebox{0.1}{%
    		\begin{tikzpicture}
       		\input{ pic_rank_11_2a.tex}
    		\end{tikzpicture}}
  	\caption{$L_4$}
  	\label{fig:pic11_2_1}
\end{subfigure}%
\begin{subfigure}{.22\textwidth}  
  	\centering
	\scalebox{0.1}{%
    		\begin{tikzpicture}
       		\input{ pic_rank_11_2b.tex}
    		\end{tikzpicture}}
  	\caption{$L_5$}
  	\label{fig:pic11_2_2}
\end{subfigure}%
\begin{subfigure}{.22\textwidth}  
  	\centering
	\scalebox{0.1}{%
    		\begin{tikzpicture}
       		\input{ pic_rank_11_2c.tex}
    		\end{tikzpicture}}
  	\caption{$L_6$}
  	\label{fig:pic11_2_3}
\end{subfigure}%
\begin{subfigure}{.22\textwidth}  
  	\centering
	\scalebox{0.1}{%
    		\begin{tikzpicture}
       		\input{ pic_rank_11_2d.tex}
    		\end{tikzpicture}}
  	\caption{$L_7$}
  	\label{fig:pic11_2_4}
\end{subfigure}%
\\[2em]
\begin{subfigure}{.22\textwidth}  
  	\centering
	\scalebox{0.1}{%
    		\begin{tikzpicture}
       		\input{ pic_rank_11_2e.tex}
    		\end{tikzpicture}}
  	\caption{$L'_4$}
  	\label{fig:pic11_2_5}
\end{subfigure}%
\begin{subfigure}{.22\textwidth}  
  	\centering
	\scalebox{0.1}{%
    		\begin{tikzpicture}
       		\input{ pic_rank_11_2f.tex}
    		\end{tikzpicture}}
  	\caption{$L'_5$}
  	\label{fig:pic11_2_6}
\end{subfigure}%
\begin{subfigure}{.22\textwidth}  
  	\centering
	\scalebox{0.1}{%
    		\begin{tikzpicture}
       		\input{ pic_rank_11_2g.tex}
    		\end{tikzpicture}}
  	\caption{$L'_6$}
  	\label{fig:pic11_2_7}
\end{subfigure}%
\begin{subfigure}{.22\textwidth}  
  	\centering
	\scalebox{0.1}{%
    		\begin{tikzpicture}
       		\input{ pic_rank_11_2h.tex}
    		\end{tikzpicture}}
  	\caption{$L'_7$}
  	\label{fig:pic11_2_8}
\end{subfigure}%
\caption{Fibrations with reducible fibers ${\color{blue} \widetilde{D}_4} + {\color{magenta} \widetilde{A}_1}^{\oplus 5}$  and {\color{yellow} section}}
\label{fig:pic11_2}
\end{figure}%

\newpage

\begin{figure}
	\centering
  	\scalebox{\MyScalePicTiny}{%
    		\begin{tikzpicture}
       			\input{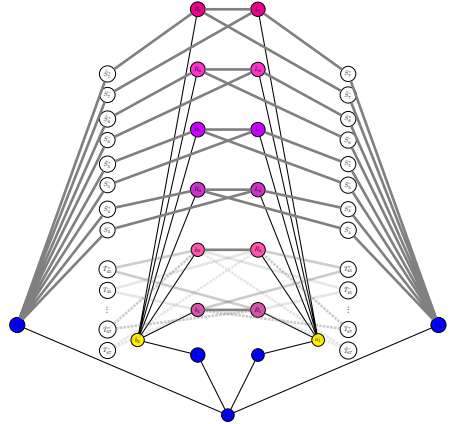}
    		\end{tikzpicture}}
	\caption{Fibration with reducible fibers ${\color{blue} \widetilde{D}_4} + {\color{magenta} \widetilde{A}_1}^{\oplus 6}$  and {\color{yellow} sections}}
  	\label{fig:pic12_1}
\end{figure}%
\begin{figure}
\begin{subfigure}{.22\textwidth}  
  	\centering
	\scalebox{\MyScalePicTiny}{%
    		\begin{tikzpicture}
       		\input{ pic_rank_12_2a.tex}
    		\end{tikzpicture}}
  	\caption{$L_4$}
  	\label{fig:pic12_2_1}
\end{subfigure}%
\begin{subfigure}{.22\textwidth}  
  	\centering
	\scalebox{\MyScalePicTiny}{%
    		\begin{tikzpicture}
       		\input{ pic_rank_12_2b.tex}
    		\end{tikzpicture}}
  	\caption{$L_5$}
  	\label{fig:pic12_2_2}
\end{subfigure}%
\begin{subfigure}{.22\textwidth}  
  	\centering
	\scalebox{\MyScalePicTiny}{%
    		\begin{tikzpicture}
       		\input{ pic_rank_12_2c.tex}
    		\end{tikzpicture}}
  	\caption{$L_6$}
  	\label{fig:pic12_2_3}
\end{subfigure}%
\begin{subfigure}{.22\textwidth}  
  	\centering
	\scalebox{\MyScalePicTiny}{%
    		\begin{tikzpicture}
       		\input{ pic_rank_12_2d.tex}
    		\end{tikzpicture}}
  	\caption{$L_7$}
  	\label{fig:pic12_2_4}
\end{subfigure}%
\caption{Fibrations with reducible fibers ${\color{blue} \widetilde{D}_4} + {\color{green} \widetilde{D}_4} + {\color{magenta} \widetilde{A}_1}^{\oplus 2}$  and {\color{yellow} section}}
\label{fig:pic12_2}
\end{figure}%
\begin{figure}
\begin{subfigure}{.44\textwidth}  
  	\centering
	\scalebox{\MyScalePicTiny}{%
    		\begin{tikzpicture}
       		\input{ pic_rank_12_3a.tex}
    		\end{tikzpicture}}
  	\caption{$L^+_2$}
\end{subfigure}%
\begin{subfigure}{.44\textwidth}  
  	\centering
	\scalebox{\MyScalePicTiny}{%
    		\begin{tikzpicture}
       		\input{ pic_rank_12_3b.tex}
    		\end{tikzpicture}}
  	\caption{$L^-_2$}
\end{subfigure}%
\caption{Fibrations with reducible fibers ${\color{blue} \widetilde{D}_6}+ {\color{magenta} \widetilde{A}_1}^{\oplus 4}$ and {\color{yellow} section}}
\label{fig:pic12_3}
\end{figure}%
\clearpage

\begin{figure}
	\centering
  	\scalebox{0.1}{%
    		\begin{tikzpicture}
       			\input{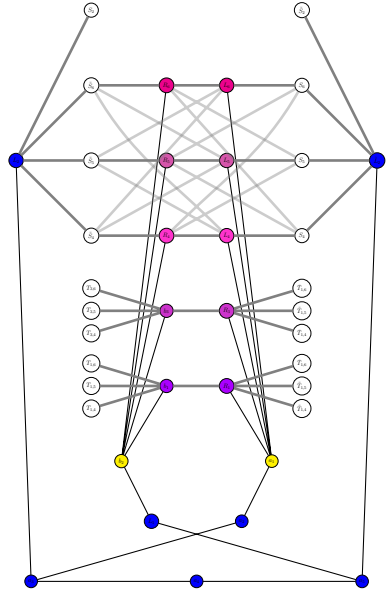}
    		\end{tikzpicture}}
	\caption{Fibration with reducible fibers ${\color{blue} \widetilde{D}_6} + {\color{magenta} \widetilde{A}_1}^{\oplus 5}$  and {\color{yellow} sections}}
	\label{fig:pic13_1}
\end{figure}%
\begin{figure}
\begin{subfigure}{.22\textwidth}  
  	\centering
	\scalebox{0.1}{%
    		\begin{tikzpicture}
       		\input{ pic_rank_13_2a.tex}
    		\end{tikzpicture}}
  	\caption{$L_4$}
  	\label{fig:pic13_2_1}
\end{subfigure}%
\begin{subfigure}{.22\textwidth}  
  	\centering
	\scalebox{0.1}{%
    		\begin{tikzpicture}
       		\input{ pic_rank_13_2b.tex}
    		\end{tikzpicture}}
  	\caption{$L_5$}
  	\label{fig:pic13_2_2}
\end{subfigure}%
\begin{subfigure}{.22\textwidth}  
  	\centering
	\scalebox{0.1}{%
    		\begin{tikzpicture}
       		\input{ pic_rank_13_2c.tex}
    		\end{tikzpicture}}
  	\caption{$L_6$}
  	\label{fig:pic13_2_3}
\end{subfigure}%
\begin{subfigure}{.22\textwidth}  
  	\centering
	\scalebox{0.1}{%
    		\begin{tikzpicture}
       		\input{ pic_rank_13_2d.tex}
    		\end{tikzpicture}}
  	\caption{$L_2$}
  	\label{fig:pic13_2_4}
\end{subfigure}%
\caption{Fibrations with reducible fibers ${\color{blue} \widetilde{D}_4} + {\color{green} \widetilde{D}_6} + {\color{magenta} \widetilde{A}_1}$  and {\color{yellow} section}}
\label{fig:pic13_2}
\end{figure}%
\begin{figure}
	\centering
  	\scalebox{0.1}{%
    		\begin{tikzpicture}
       			\input{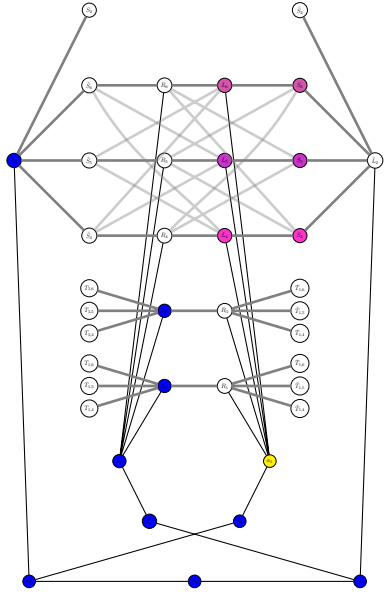}
    		\end{tikzpicture}}
	\caption{Fibration with reducible fibers ${\color{blue} \widetilde{D}_8} + {\color{magenta} \widetilde{A}_1}^{\oplus 3}$  and {\color{yellow} section}}
	\label{fig:pic13_3}
\end{figure}%
\begin{figure}
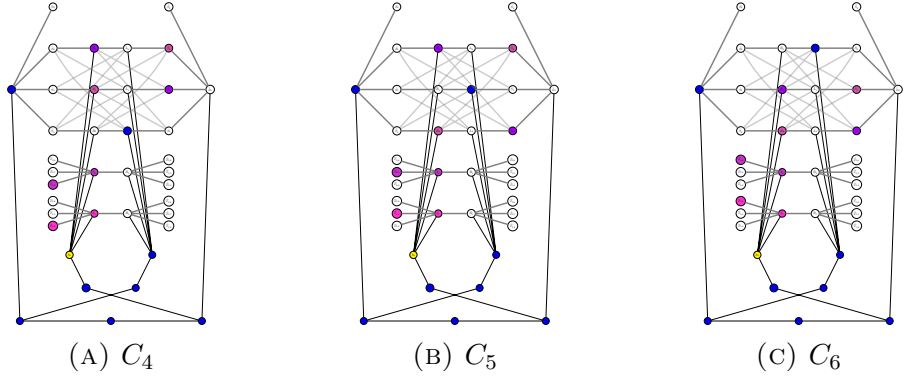

\begin{subfigure}{.3\textwidth}  
  	\centering
	\scalebox{0.1}{%
    		\begin{tikzpicture}
       		\input{ pic_rank_13_4a.tex}
    		\end{tikzpicture}}
  	\caption{$C_4$}
  	\label{fig:pic13_4_1}
\end{subfigure}%
\begin{subfigure}{.3\textwidth}  
  	\centering
	\scalebox{0.1}{%
    		\begin{tikzpicture}
       		\input{ pic_rank_13_4b.tex}
    		\end{tikzpicture}}
  	\caption{$C_5$}
  	\label{fig:pic13_4_2}
\end{subfigure}%
\begin{subfigure}{.3\textwidth}  
  	\centering
	\scalebox{0.1}{%
    		\begin{tikzpicture}
       		\input{ pic_rank_13_4c.tex}
    		\end{tikzpicture}}
  	\caption{$C_6$}
  	\label{fig:pic13_4_3}
\end{subfigure}%
\caption{Fibration with reducible fibers ${\color{blue} \widetilde{E}_7}+ {\color{magenta} \widetilde{A}_1}^{\oplus 4}$  and {\color{yellow} section}}
\label{fig:pic13_4}
\end{figure}%
\clearpage
%
%
\section*{Data Availability Statement}
This manuscript has no associated data.
\section*{Conflict of Interest Statement}
On behalf of all authors, the corresponding author states that there is no conflict of interest.
%
%
\bibliographystyle{amsplain}
\bibliography{references}{}
\end{document}